\documentclass{cmslatex}
\usepackage[paperwidth=7in, paperheight=10in, margin=.875in]{geometry}
\usepackage[colorlinks,linkcolor=black,anchorcolor=black,citecolor=blue]{hyperref}
\usepackage{amsfonts,amssymb}
\usepackage{amsmath}
\usepackage{graphicx}
\usepackage{cite}
\usepackage{enumerate}
\renewcommand{\theequation}{\arabic{section}.\arabic{equation}}

\makeatletter
\newcommand*{\NoBreakPar}{\vspace{\baselineskip}\par\nobreak\@afterheading}
\makeatother

\usepackage{cleveref}
\usepackage{xspace}
\usepackage{subcaption}

\usepackage{floatrow}
\newfloatcommand{capbtabbox}{table}[][\FBwidth]

\crefname{equation}{}{}

\newcommand{\M}{\ensuremath{\mathbf{M}}\xspace}
\newcommand{\m}{\ensuremath{\mathbf{m}}\xspace}

\renewcommand{\H}{\ensuremath{\mathbf{H}}\xspace}
\newcommand{\h}{\ensuremath{\mathbf{h}}\xspace}

\renewcommand{\P}{\ensuremath{\mathbf{P}}\xspace}
\newcommand{\Q}{\ensuremath{\mathbf{Q}}\xspace}

\newcommand{\Real}{\ensuremath{\mathbb{R}}\xspace}

\newcommand{\F}{\ensuremath{\mathbf{F}}\xspace}
\newcommand{\f}{\ensuremath{\mathbf{f}}\xspace}
\newcommand{\A}{\ensuremath{\mathbf{A}}\xspace}
\newcommand{\I}{\ensuremath{\mathbf{I}}\xspace}
\newcommand{\Y}{\ensuremath{\mathbf{Y}}\xspace}
\newcommand{\T}{\ensuremath{\mathbf{T}}\xspace}
\renewcommand{\k}{\ensuremath{\mathbf{k}}\xspace}
\newcommand{\n}{\ensuremath{\mathbf{n}}\xspace}

\newcommand{\bnabla}{\boldsymbol\nabla}

\begin{document}

\title{Heterogeneous multiscale methods for the Landau-Lifshitz
  equation} \author{Lena Leitenmaier \thanks{Department of
    Mathematics, KTH, Royal Institute of Technology, Stockholm,
    Sweden, (lenalei@kth.se).} \and Olof Runborg \thanks{Department
    of Mathematics, KTH, Royal Institute of Technology, Stockholm,
    Sweden, (olofr@kth.se).}}

 \pagestyle{myheadings} \markboth{HMM for the
   Landau-Lifshitz equation}{L. Leitenmaier and
   O. Runborg} \maketitle


\begin{abstract}
  In this paper, we present a finite difference heterogeneous
  multiscale method for the Landau-Lifshitz equation with a highly
  oscillatory diffusion coefficient. The approach combines a higher
  order discretization and artificial damping in the so-called micro
  problem to obtain an efficient implementation.  The influence of
  different parameters on the resulting approximation error is
  discussed.  Numerical examples for both periodic as well as more
  general coefficients are given to demonstrate the functionality of
  the approach.
\end{abstract}



\section{Introduction}
The simulation of ferromagnetic composites can play an important
role in the development of magnetic materials. A typical approach to
describing magnetization dynamics of ferromagnetic materials is using
the micromagnetic version of the Landau-Lifshitz equation, which
states
\begin{align}\label{eq:LL}
  \partial_t \M^\varepsilon = - \M^\varepsilon \times \H(\M^\varepsilon) - \alpha \M^\varepsilon \times \M^\varepsilon \times \H(\M^\varepsilon),
\end{align}
where $\M^\varepsilon$ is the magnetization vector,
$\H(\M^\varepsilon)$ the effective field acting on the magnetization
and the material constant $\alpha$ describes the strength of
damping. While the effective field contains several important
contributions, we here consider a simplified model, only taking into
account exchange interaction, and introduce a coefficient
$a^\varepsilon$ describing the material variations in the
composite. The parameter $\varepsilon \ll 1$ represents the scale of
these variations. We then have
\[\H(\M^\varepsilon) := \bnabla \cdot (a^\varepsilon \bnabla  \M^\varepsilon),\]
a model that was first described in \cite{hamdache}. Similar
models have also recently been used by for example \cite{alouges2019stochastic}
as well as  \cite{multilayer} and \cite{highcontrast}.

For small values of $\varepsilon$, it becomes computationally very
expensive and at some point infeasible to provide proper numerical
resolution for a simulation of \cref{eq:LL}. Hence we aim to apply
numerical homogenization based on the approach of Heterogeneous
Multiscale Methods (HMM) to the problem. In this framework, one
combines a coarse scale macro problem with a micro problem resolving
the relevant fast scales on a small domain in order to obtain an approximation to the
effective solution corresponding to the problem.

For a simplified Landau-Lifshitz problem with a highly oscillatory external
field and no spatial interaction, a possible HMM setup was
introduced in \cite{arjmand1} and extended to a non-zero temperature
scenario in \cite{arjmand2}.

For the problem we consider here, \cref{eq:LL}, the homogenization
error has been analyzed in \cite{paper1}. There it is also shown
that $\M^\varepsilon$ exhibits fast oscillations in both space and
time, where the spatial variations are of order
$\mathcal{O}(\varepsilon)$ while the temporal ones are of order
$\mathcal{O}(\varepsilon^2)$ and get damped away exponentially with
time, depending on the given value of $\alpha$. In \cite{paper2}
several ways to set up HMM were discussed and the errors introduced
in the numerical homogenization process, the so-called upscaling
errors, where analyzed. In this paper, we focus on numerical aspects
related to the implementation of HMM for \cref{eq:LL}.  In
\Cref{sec:hmm}, we first give an overview of the method and include
relevant known results from \cite{paper1} and \cite{paper2}. We then
discuss some aspects of time integration of the Landau-Lifshitz
equation in \Cref{sec:ts} and suggest suitable methods for the time
stepping in the macro and micro problem, respectively.
\Cref{sec:micro} focuses on the HMM micro problem. We study how to
choose initial data for the micro problem that is appropriately
coupled to the current macro scale solution in \Cref{sec:initial},
before using numerical example problems to investigate several
factors that influence the errors introduced in the HMM averaging
process in \Cref{sec:micro_size}.  In \Cref{sec:num_ex} we present
numerical examples to show that the HMM approach can also be applied
to locally-periodic and quasi-periodic problems.

\section{Heterogeneous Multiscale Methods}\label{sec:hmm}

In this section, we introduce the concept of Heterogeneous
Multiscale Methods, discuss how we choose to set up a HMM model for the
Landau-Lifshitz problem \cref{eq:LL} and give relevant error estimates that were
introduced in \cite{paper1} and \cite{paper2}.

\subsection{Problem description}
The specific multiscale problem we consider in this article is to find $\M^\varepsilon$
that satisfies the nonlinear initial value problem
\begin{subequations}\label{eq:prob}
\begin{align}
  \partial_t \M^\varepsilon &= - \M^\varepsilon \times \bnabla \cdot(a^\varepsilon \bnabla \M^\varepsilon) - \alpha \M^\varepsilon \times \M^\varepsilon \times \bnabla \cdot(a^\varepsilon \bnabla \M^\varepsilon), \\
  \M^\varepsilon(x, 0) &= \M_\mathrm{init},
\end{align}
\end{subequations}
  with periodic boundary conditions, on a fixed time interval
  $[0, T]$ and a spatial domain $\Omega = [0, L]^d$ for some
  $L \in \mathbb{N}$ and dimension $d = 1, 2, 3$.  Here $a^\varepsilon$ is a
  material coefficient which oscillates with a frequency determined
  by $\varepsilon$.  We furthermore assume the following.
\begin{itemize}
\item [(A1)] The material coefficient function $a^\varepsilon$ is in
  $C^\infty(\Omega)$ and bounded by constants
  $a_\mathrm{min}, a_\mathrm{max} > 0$; it holds that
  $a_\mathrm{min}\leq a^\varepsilon(x)\leq a_\mathrm{max}$ for all $x \in \Omega$.

\item [(A2)] The damping coefficient $\alpha$ and the oscillation period
$\varepsilon$ are small, $0 < \alpha \le 1$ and $0 < \varepsilon < 1$.
\item [(A3)] The initial data $\M_\mathrm{init}(x)$ is such that
  $| \M_\mathrm{init}(x)| = 1$ for all $x \in \Omega$, which implies
  that $|\M^\varepsilon(x, t)| = 1$ for all $x \in \Omega$ and
  $t \in [0, T]$.
\end{itemize}

When the material coefficient is periodic,
$a^\varepsilon = a(x/\varepsilon)$ where
$\varepsilon = L/\ell$ for some $\ell \in \mathbb{N}$, one
can analytically derive a homogenized problem corresponding to
\cref{eq:prob}, as shown in \cite{paper1}. The solution $\M_0$ to
this homogenized problem satisfies
\begin{subequations}\label{eq:hom}
\begin{align}
  \partial_t \M_0 &= - \M_0 \times \bnabla \cdot(\bnabla \M_0 \A^H) - \alpha \M_0 \times \M_0 \times \bnabla \cdot(\bnabla \M_0 \A^H), \\
  \M_0(x, 0) &= \M_\mathrm{init},
\end{align}
\end{subequations}
where $\A^H$ is the same homogenized coefficient matrix as for
standard elliptic homogenization problems,
\begin{align}\label{eq:AH}
  \A^H := \int_Y a(y) \left( \I + (\bnabla_y \boldsymbol \chi)^T \right) dy\,.
\end{align}
Here $\boldsymbol\chi(y) \in \Real^d$ denotes the so-called cell
solution, which satisfies
\begin{equation}
  \label{eq:cell_problem}
  \bnabla \cdot( a(y) \bnabla \boldsymbol \chi(y)) = - \nabla_y a(y) \,
\end{equation}
and is defined to have zero average. In \cite{paper1}, error bounds
for the difference between the solutions to \cref{eq:prob} and
\cref{eq:hom} are proved under certain regularity assumptions. In
particular,
we have the following result for periodic problems.

\begin{theorem}\label{thm:paper1}
  Given a fixed final time $T$, assume that
  $\M^\varepsilon \in C^1([0, T]; H^{2}(\Omega))$ is a classical
  solution to \cref{eq:prob} and that there is a constant $K$
  independent of $\varepsilon$ such that
  $\|\bnabla \M^\varepsilon(\cdot, t)\|_{L^\infty} \le K$ for all
  $t \in [0, T]$. Suppose that
  $\M_0 \in C^\infty(0, T; H^{\infty}(\Omega))$ is a classical
  solution to \cref{eq:hom} and that the assumptions (A1)-(A3) are
  satisfied.
  We then have for $0 \le t \le T$,
  \begin{align}\label{eq:main_l2_m0}
  \|\M^\varepsilon(\cdot, t) - \M_0(\cdot, t)\|_{L^2} \le C \varepsilon\,,
  \end{align}
where the constant $C$ is independent of $\varepsilon$ and $t$ but
depends on $K$ and $T$.
\end{theorem}

Note that it is easy to show that
$\|\bnabla \M^\varepsilon(\cdot, t)\|_{L^2} \le K$ independent of
$\varepsilon$, as is for example shown in \cite[Appendix B]{paper1}.
Numerically one can check that the same also holds for
$\|\bnabla \M^\varepsilon(\cdot, t)\|_{L^\infty}$.


\subsection{Heterogeneous Multiscale Methods for the Landau-Lifshitz equation}
%
%
Heterogeneous Multiscale Methods are a well-established framework
for dealing with multiscale problems with scale separation, that
involve fast scale oscillations which make it computationally infeasible
to properly resolve the problem throughout the whole domain. First
introduced by E and Engquist \cite{weinan2003}, they have since then
been applied to problems from many different areas \cite{acta_numerica, hmm1}.

The general idea of HMM is to approximate the effective solution to
the given problem using a coarse scale macro model that is missing
some data and is thus incomplete. It is combined with an accurate
micro model resolving the fast oscillations in the problem, coupled
to the macro solution via the micro initial data. The micro problem
is only solved on a small domain around each discrete macro location
to keep the computational cost low.  The thereby obtained solution
is then averaged and provides the information necessary to complete
the macro model \cite{weinan2003, acta_numerica, hmm1}.


Since HMM approximates the effective solution to a multiscale
problem rather than resolving the fast scales, some error is
introduced. In case of the Landau-Lifshitz problem \cref{eq:prob}
with a periodic material coefficient, the effective solution
corresponding to $\M^\varepsilon$ is $\M_0$ satisfying
\cref{eq:hom}. It hence follows from \Cref{thm:paper1} that the
$L^2$-error between the HMM solution and $\M^\varepsilon$ in the
periodic case is always at least $\mathcal{O}(\varepsilon)$.

There are several different HMM models one could choose for the
problem \cref{eq:prob}. Three possibilities are discussed in
\cite{paper2}, flux, field and torque model. All three are based on
the same micro model, the full Landau-Lifshitz equation
\cref{eq:prob} which is solved on a time interval $[0, \eta]$, where
$\eta \sim \varepsilon^2$. In \cite{paper1}, it is shown that this
is the scale of the fast temporal oscillations in the problem.
Hence, the micro model is to find $\m^\varepsilon(x, t)$ for
$0 \le t \le \eta$ such that
\begin{subequations}\label{eq:micro}
\begin{align}
  \partial_t \m^\varepsilon &= - \m^\varepsilon \times \bnabla \cdot(a^\varepsilon \bnabla \m^\varepsilon) - \alpha \m^\varepsilon \times \m^\varepsilon \times \bnabla \cdot(a^\varepsilon \bnabla \m^\varepsilon), \\
  \m^\varepsilon(x, 0) &= \m_\mathrm{init}(x) = \Pi^k \M(\cdot, t_j).
\end{align}
\end{subequations}
The initial data for the micro problem is based on an interpolation
of the current macro state $\M$, here denoted by $\Pi^k$, which is
explained in more detail in \Cref{sec:initial}.  In \cite{paper2},
it is assumed that \cref{eq:micro} holds for $x \in \Omega$ with
periodic boundary conditions to simplify the analysis. In practice,
one must only solve \cref{eq:micro} for $x \in [-\mu', \mu']^d$ to keep down the computational cost.
Here $\mu' \sim \varepsilon$, since this is the scale of the fast
spatial oscillations in $\m^\varepsilon$. To do this, we have to add
artificial boundary conditions which introduce some error as
discussed in \Cref{sec:bc,sec:micro_size}.

The three different macro models considered in \cite{paper2} have
the general structure of \cref{eq:prob,eq:hom} but involve
different unknown quantities which have to be obtained by averaging
the corresponding data from the micro model \cref{eq:micro}. In the field model,
we have
\begin{subequations}\label{eq:macro}
  \begin{align}
    \partial_t \M &= - \M \times \H_\mathrm{avg}(x, t; \M) - \alpha \M \times \M \times \H_\mathrm{avg}(x, t; \M), \label{eq:macro_field}\\
    \M(x, 0) &= \M_\mathrm{init},
  \end{align}
\end{subequations}
where $\H_\mathrm{avg}(x, t; \M)$ denotes the unknown quantity. In
the periodic case, this quantity approximates
$\bnabla \cdot (\bnabla \m \A^H)$.

In the flux model, \cref{eq:macro_field} is replaced by
  \begin{align}\label{eq:macro_flux}
    \partial_t \M &= - \M \times \bnabla \cdot \F_\mathrm{avg}(x, t; \M) - \alpha \M \times \M \times \F_\mathrm{avg}(x, t; \M),
  \end{align}
  and in the torque model, we instead have
 \begin{align}\label{eq:macro_torque}
    \partial_t \M &= - \T_\mathrm{avg}(x, t; \M) - \alpha \M \times \T_\mathrm{avg}(x, t; \M),
  \end{align}
  where in the periodic case, $\F_\mathrm{avg}$ and
  $\T_\mathrm{avg}$ are approximations to $\bnabla \m \A^H$ and
  $\m \times \bnabla \cdot (\bnabla \m \A^H)$, respectively.  As
  shown in \cite{paper2}, the error introduced when approximating the
  respective quantities by an averaging procedure, the so-called
  upscaling error, is bounded rather similarly for all three models,
  with somewhat lower errors in the flux model. This does not give a
  strong incentive to choose one of the models over the others. In
  this paper, we therefore focus on the field model for the following reasons,
  not related to the upscaling error.

  First, when choosing the flux model, the components of the flux
  should be approximated at different grid locations to reduce the
  approximation error in the divergence that has to be computed on
  the macro scale, which typically has a rather coarse
  discretization. This implies that we need to run separate micro
  problems for each component of the gradient. This is not necessary
  when using the field model.

  Second, it is seen as an important aspect of micromagnetic
  algorithms that the norm preservation property of the continuous
  Landau-Lifshitz equation is mimicked by time integrators for the
  discretized problem. This is usually achieved by making use of the
  cross product structure in the equation. However, when choosing
  the torque model \cref{eq:macro_torque}, there is no cross product
  in the first term of the macro model.

  The chosen HMM macro model, \cref{eq:macro}, is discretized on a
  coarse grid in space with grid spacing $\Delta X$ and points
  $x_i = x_0 + i \Delta X$, where $i$ is a $d$-dimensional
  multi-index ranging from 0 to $N$ in each coordinate direction.
  The corresponding semi-discrete magnetization values are $\M_i(t) \approx \M(x_i, t)$, which
  satisfy the semi-discrete equation
  \begin{subequations}\label{eq:macro_disc}
    \begin{align}
      \partial_t \M_i &= - \M_i \times \H_\mathrm{avg}(x_i, t;  \bar \M) - \alpha \M_i \times \M_i \times \H_\mathrm{avg}(x_i, t;  \bar \M), \\
      \M_i(0) &= \M_\mathrm{init}(x_i),
    \end{align}
  \end{subequations}
  where $ \bar \M$ denotes the vector containing all
  the $\M_i$, $i \in \{0, ...,  N\}^d$.  The notation
  $\H_\mathrm{avg}(x_i, t;  \bar \M)$ represents the
  dependence of $\H_\mathrm{avg}$ at location $x_i$ on several values of the discrete
  magnetization at time $t$.  To
  discretize \cref{eq:macro_disc} in time, we introduce
  $t_j = t_0 + j \Delta t$, for $j = 0, ..., M$.  The specific form
  of time discretization of \cref{eq:macro_disc} is discussed in
  \Cref{sec:ts}.

  \subsection{Upscaling}\label{sec:upscaling}
  To approximate the unknown quantity $\H_\mathrm{avg}$ in
  \cref{eq:macro_disc} in an efficient way and to control how fast the
  approximation converges to the corresponding effective quantity,
  we use averaging involving kernels as introduced in \cite{stiff,
    doghonay1}.

\begin{definition}[\cite{doghonay1,  paper2}]\label{def:kernel}
A function $K$ is in the space of smoothing kernels $\mathbb{K}^{p, q}$ if
\begin{enumerate}
\item $K \in C_c^{q}([-1, 1])$ and  $K^{(q+1)} \in BV(\Real)$ .
\item $K$ has $p$ vanishing moments,
  \[\int_{-1}^1 K(x) x^r dx =
  \begin{cases}
    1\,, & r = 0\,,\\
    0\,, & 1 \le r \le p\,.
  \end{cases}\]
\end{enumerate}
If additionally $K(x) = 0$ for $x \le 0$ then $K \in \mathbb{K}_0^{p, q}$.
\end{definition}

We use the conventions that $K_\mu$ denotes a scaled version of the kernel $K$,
\[K_\mu(x) := {1}/{\mu} K(x/\mu),\]
and that in space dimensions with $d > 1$,
\[K(x) := K(x_1) \cdots K(x_d).\]
For the given problem, we choose a kernel
$K \in \mathbb{K}^{p_x, q_x}$ for the spatial and
$K^0 \in \mathbb{K}^{p_t, q_t}_0$ for the temporal averaging due to the
fact that \cref{eq:micro} cannot be solved backward in time. The
particular upscaling procedure at time $t_j$ is then given by
\begin{align}\label{eq:upscaling}
  \H_\mathrm{avg}(x_i, t_j; \bar \M) = \int_0^\eta \int_{\Omega_\mu} K_\mu(x) K_\eta^0(t) \bnabla \cdot (a^\varepsilon \bnabla \m^\varepsilon) dx dt,
\end{align}
where $\Omega_\mu := [-\mu, \mu]^d$ for a parameter
$\mu \sim \varepsilon$ such that $\mu \le \mu'$, the averaging
domain is a subset of the domain that the micro problem is solved
on. The micro solution $\m^\varepsilon$ is obtained solving
\cref{eq:micro} on $[-\mu', \mu']^d \times [0, \eta]$, with initial
data $\m_\mathrm{init}$ based on $\bar \M$ at time $t_j$ and around
the discrete location $x_i$.  A second order central difference
scheme in space is usually sufficient to obtain an approximation to $\m^\varepsilon$
with errors that are low compared to the averaging errors at a
relatively low computational cost.

Assuming instead that the micro problem is solved on
$\Omega \times [0, \eta]$, we have the following estimate for the
upscaling error for the case of a periodic material coefficient that
is proved in \cite{paper2}.

\begin{theorem}\label{thm:paper2}
  Assume that (A1)-(A2) hold and the micro initial data $\m_\mathrm{init}$ is such that it satisfies
  (A3).  Let  $\varepsilon^2 < \eta \le \varepsilon^{3/2}$ and suppose that for
  $x \in \Omega$ and $0 \le t \le \eta$, the exact solution to the
  micro problem \cref{eq:micro} is
  $\m^\varepsilon(x, t) \in C^1([0, \eta]; H^{2}(\Omega))$ and that
  there is a constant $c$ independent of $\varepsilon$ such that
  $\|\bnabla \m^\varepsilon(\cdot, t)\|_{L^\infty} \le c$.  The
  solution to the corresponding homogenized problem is
  $\m_0 \in C^\infty(0, \eta, H^\infty(\Omega))$.
  Moreover, consider averaging kernels   $K \in \mathbb{K}^{p_x, q_x}$ and
  $K^0 \in \mathbb{K}^{p_t, q_t}_0$ and let  $\varepsilon < \mu < 1$.
  Then
  \begin{align*}
    &\left| \H_\mathrm{avg} - \bnabla \cdot (\bnabla \m_\mathrm{init}(0) \A^H) \right|
      =: E_\varepsilon + E_\mu + E_\eta,
  \end{align*}
  where
  \begin{align}\label{eq:err_terms}
    E_\varepsilon \le C \varepsilon, \quad E_\mu \le C \left(\mu^{p_x +1} +  \left(\frac{\varepsilon}{\mu}\right)^{q_x + 2}\right) \quad \text{and} \quad E_\eta \le C \left( \eta^{p_t + 1} + \frac{1}{{\mu}}  \left(\frac{\varepsilon^2}{ \eta}\right)^{q_t+1}\right).
  \end{align}
  In all cases, the constant $C$ is independent of
  $\varepsilon$, $\mu$ and $\eta$ but might depend on $K$, $K^0$ and
  $\alpha$.
\end{theorem}

As discussed in \cite{paper2}, for periodic problems we in practice
often observe $E_\varepsilon = \mathcal{O}(\varepsilon^2)$ rather
than the more pessimistic estimate in the theorem.

Two things are important to note here.  First, \Cref{thm:paper2}
states that $\H_\mathrm{avg}(x_i, t_j; \bar \M)$ approximates the
solution to the corresponding effective quantity involving the micro
scale initial data $\m_\mathrm{init}$, not the exact macro solution
$\M$.  We therefore have to require that
\begin{align}\label{eq:consistency_cond}
  \partial_x^\beta \m_\mathrm{init}(0) = \partial_x^\beta \M(x_i, t_j)
\end{align}
for a multi-index $\beta$ with $|\beta| = 2$ to get an estimate for
the actual upscaling error.  Bearing in mind a somewhat more general
scenario with non-periodic material coefficient where $\A^H$ no
longer is constant, we subsequently
require \cref{eq:consistency_cond} to hold for $|\beta| \le 2$.

Moreover, the quantity that $\H_\mathrm{avg}$ in \Cref{thm:paper2}
approximates is independent of $\alpha$.  We can thus choose a
different damping parameter in the micro problem than in the macro
one to optimize the constants in \cref{eq:err_terms}. Typically, it is
favorable to have higher damping in the micro problem
\cref{eq:micro}, as is discussed in the following sections. This can
be seen as an introduction of artificial damping to improve
numerical properties as is common in for example hyperbolic
problems.

\subsection{Example problems}
Throughout this article, we use three different periodic example problems to
illustrate the behavior of the different HMM components and
numerical methods under discussion, one 1D example and two 2D examples.
Further, non-periodic examples are discussed in \Cref{sec:num_ex}.
\begin{itemize}
\item [(EX1)]
For the 1D example, the initial data is chosen to be
 \begin{align*}
  \M_\mathrm{init}(x) = \tilde \M(x)/|\tilde \M(x)|, \qquad \tilde \M(x) =
  \begin{bmatrix}
    0.5+\exp(-0.1\cos(2\pi(x-0.32))) \\
    0.5+\exp(-0.2\cos(2\pi x)) \\
    0.5+\exp(-0.1\cos(2\pi(x-0.75)))
  \end{bmatrix}
\end{align*}
and the material coefficient we consider is
$a^\varepsilon(x) = a(x/\varepsilon)$ where
\[a(x) = 1 + 0.5\sin(2 \pi x).\] The corresponding homogenized
coefficient, which is not used in the HMM approach but as a
reference solution, is
$\A^H = \left(\int_0^1 1/a(x) dx \right)^{-1} \approx 0.866$.  In
\Cref{fig:ex1_sol} the solution $\M^\varepsilon(x, T)$ at $T = 0.1$
and the corresponding HMM approximation $\M(x, T)$ computed on a
grid with $\Delta X = 1/24$ are shown. Note that the HMM approximation agrees
very well with $\M^\varepsilon$.

\begin{figure}[H]
  \centering
  \includegraphics[width=.8\textwidth]{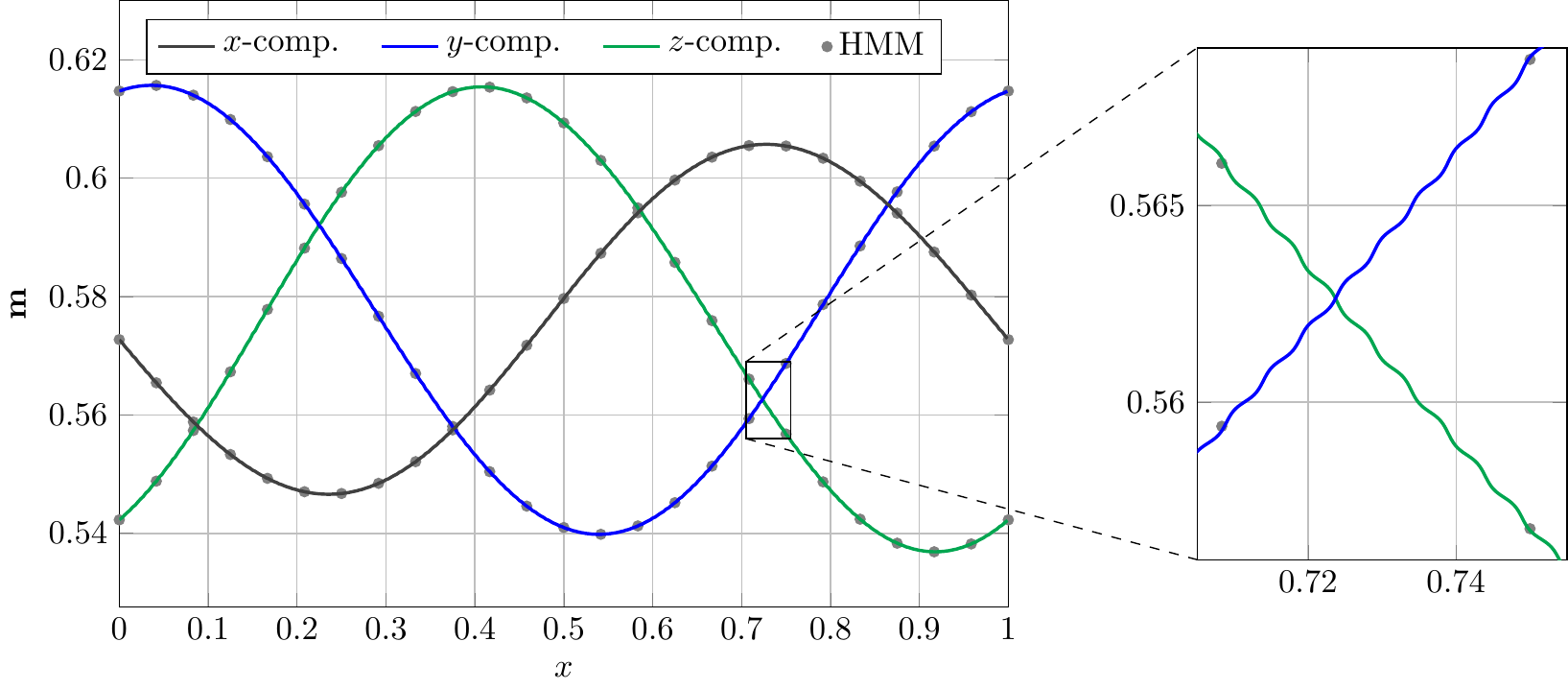}
  \caption{Solution $\m^\varepsilon$ and corresponding HMM
    approximation to \cref{eq:LL} with setup (EX1) at time $T = 0.1$
    when $\varepsilon = 1/200$.}
  \label{fig:ex1_sol}
\end{figure}

\item [(EX2)] For the first 2D example the initial data is
 \begin{align*}
     \m_\mathrm{init}(x) &= \tilde \m(x)/|\tilde \m(x)|, \\
     \tilde \m(x) &=
  \begin{bmatrix}
    0.6+\exp(-0.3(\cos(2\pi(x_1-0.25)) + \cos(2\pi(x_2-0.12)))) \\
    0.5+\exp(-0.4(\cos(2\pi x_1) + \cos(2\pi (x_2-0.4))))  \\
    0.4+\exp(-0.2(\cos(2\pi(x_1-0.81)) + \cos(2\pi(x_2 - 0.73))))
  \end{bmatrix},
 \end{align*}
 which is shown in \cref{fig:init_2D}.
 The material coefficient is given by
   \begin{align*}
     a(x) &= 0.5 + (0.5 + 0.25\sin(2 \pi x_1))(0.5 + 0.25\sin(2 \pi x_2)) \\
		&\qquad + 0.25 (\cos(2 \pi (x_1-x_2)) + \sin(2 \pi x_1)),
   \end{align*}
   which corresponds to a homogenized coefficient with non-zero off-diagonal elements,
   \[\A^H \approx
     \begin{bmatrix}
       0.617&   0.026 \\
       0.026&   0.715
     \end{bmatrix}.\]
   For this example, $\A^H$ has to be computed numerically (with high precision).

   \begin{figure}[h]
  \centering
  \includegraphics[width=\textwidth]{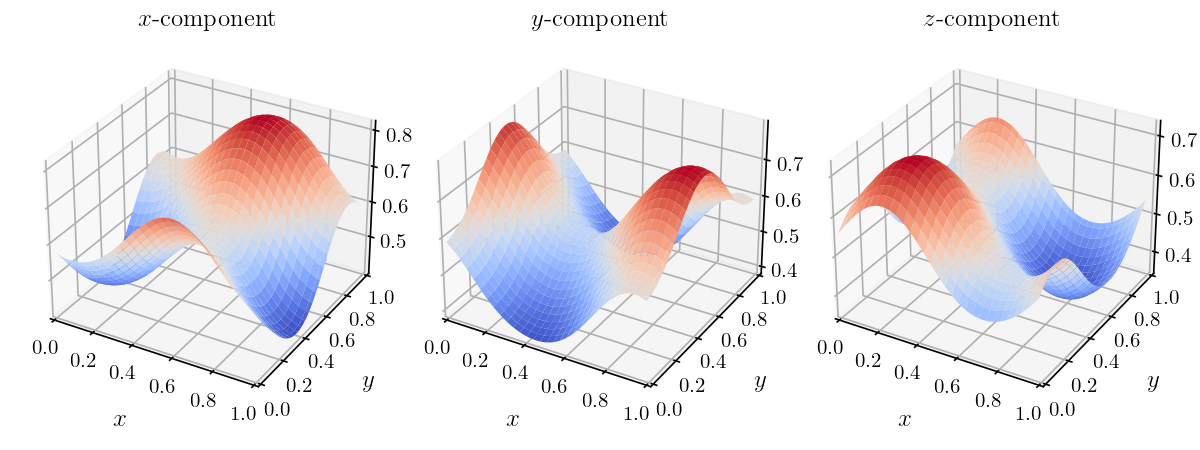}
  \caption{Initial data $\M_\mathrm{init}$ for the 2D problems}
  \label{fig:init_2D}
\end{figure}

 \item [(EX3)] The second 2D example has the same initial data as
   (EX2) but a different material coefficient,
 \[a(x) = (1.1 + 0.5\sin(2\pi x_1))(1.1 + 0.5\sin(2 \pi
     x_2)).\]
   The corresponding homogenized matrix can be computed
   analytically \cite{doghonay2} and takes the value
   \[\A^H = 1.1 \sqrt{1.1^2 - 0.25} \I.\]

\end{itemize}
In all three cases, it holds that $a^\varepsilon(x) = a(x/\varepsilon)$.

\section{Time stepping for Landau-Lifshitz problems}\label{sec:ts}

A variety of different methods for time integration of the
Landau-Lifshitz equation in a finite difference setting are
available, as for example discussed in the review articles
\cite{garcia_review, cimrak2007survey, llg_lncs}.  Most of these
methods can be characterized as either projection methods or
geometric integrators, typically based on the implicit midpoint
method. In a projection method, the basic update procedure does not
preserve the length of the magnetization vector moments which makes
it necessary to project the intermediate result back to the unit
sphere at the end of each time step.  Commonly used examples for
this kind of methods are projection versions of Runge-Kutta methods
as well as the Gauss-Sequel projection method \cite{gspm,
  gspm_imp}. Furthermore, an implicit projection method based on a
linear update formula is proposed in \cite{weinan_LL}.

The most common geometric integrator is the implicit midpoint
method, which is both norm preserving and in case of no damping,
$\alpha = 0$, also energy conserving
\cite{dAquino2005geometrical, dAquino2005numerical}. However, since
it is computationally rather expensive, several semi-implicit
variations have been proposed, in particular SIA and SIB introduced
in \cite{mentink} as well as the midpoint extrapolation method, MPE,
\cite{mpe}. Further geometric integrators are the Cayley transform
based approaches discussed in \cite{lewis_nigam,
  krishnaprasad2001cayley}.

While there are many methods available for time integration of the
Landau-Lifshitz equation, which have advantages in different
scenarios, we here have a strong focus on computational cost,
especially when considering the micro problem, where the subsequent
averaging process reduces the importance of conservation of physical
properties.
For the HMM macro model, the form of the problem, \cref{eq:macro},
prevents the rewriting of the equation as would be necessary for some
integrators, for example the method in \cite{weinan_LL}. In general,
the dependence of $\H_\mathrm{avg}$ in \cref{eq:macro} on the micro
solution makes the use of implicit methods problematic, as is
further discussed in \Cref{sec:macro_ts}.

In the following, we focus on several time integration methods that
might be suitable for the given setup and then motivate our choice
for the macro and micro problem, respectively. %

\subsection{Description of selected methods}\label{sec:description}

The methods we focus on are two projection methods, HeunP and RK4P,
as well as the semi-implicit midpoint extrapolation method, MPE,
introduced in \cite{mpe} and MPEA, an adaption of the latter
method.  We furthermore
include the implicit midpoint method in the considerations since it
can be seen as a reference method for time integration of the
Landau-Lifshitz equation. 

In this section, we suppose that we work with a discrete grid in
space with locations $x_i = x_0 + i \Delta x$, where
$i \in \{0, ..., N\}^d$, and consider time points $t_j = t_0 + j \Delta t$, $j = 0, ..., M$.
We denote by $\m_i^j \in \Real^3$ an approximation to the
magnetization at location $x_i$ and time $t_j$,
$\m_i^j \approx \m(x_i, t_j)$.  When writing $\m^j$ we refer to a
vector in $\Real^{3N^d}$ that contains all $\m_i^j$.
In the main part of this section, we do not distinguish
between macro and micro problem but focus on the general behavior
of the time stepping methods. Thus, $\m$ can denote both a micro or
macro solution.

We furthermore use the notation $\f_i(\m^j)$ to denote the value of
a function $\f$ at location $x_i$ which might depend on the values
of $\m^j$ at several space locations.  In particular, we write
$\H_i(\m^j)$ to denote a discrete approximation to the  effective field at $x_i$. 
%
On the macro scale, it thus holds that
$\H_i(\m^j) \approx \H_\mathrm{avg}(x_i, t_j; \bar \M)$ or, when
considering the corresponding homogenized problem in case of a periodic material coefficient,
\[\H_i(\m^j) \approx \bnabla \cdot(\bnabla \m(x_i, t_j) \A^H).\]
 For the micro problem, we have
\[\H_i(\m^j) \approx \bnabla \cdot(a(x_i/\varepsilon) \bnabla
  \m^\varepsilon(x_i, t_j)).\]
The particular form of $\H_i(\m^j)$
does not have a major influence on the following discussions if not
explicitly stated otherwise.

\subsubsection*{HeunP and RK4P}
HeunP and RK4P are the standard Runge Kutta 2 and Runge Kutta 4
methods with an additional projection back to the unit sphere at the
end of every time step.  Let $\f(\m^j)$ be the ${ 3N^d}$-vector such that
\[\f_i(\m^j) := - \m_i^j \times \H_i(\m^j) - \alpha \m_i^j \times \m_i^j \times
  \H_i(\m^j), \qquad i \in \{0, ..., N\}^d.\]
Then in the Runge-Kutta methods, one computes stage values
\begin{align*}
  \k_1 = \f(\m^j), \quad \k_2 = \f(\m^j + \tfrac{\Delta t}{2} \k_1), \quad \k_2 = \f(\m^j + \tfrac{\Delta t}{2} \k_2), \quad \k_4 = \f(\m^j + \Delta t \k_3).
\end{align*}
In HeunP (RK2P), the time step update then is given by
\begin{align}
  \m^{j+1}_i = \tilde \m_i / |\tilde \m_i|, \qquad \text{where} \qquad   \tilde \m = \m^j + \tfrac{\Delta t}{2}(\k_1 + \k_2),
\end{align}
and for RK4P,
\begin{align}
  \m^{j+1}_i = \tilde \m_i / |\tilde \m_i|, \qquad \text{where} \qquad   \tilde \m = \m^j + \tfrac{\Delta t}{6}(\k_1 + 2 \k_2 + 2 \k_3 + \k_4).
\end{align}
HeunP is a second order method and RK4P is fourth order accurate.

\subsubsection*{Implicit midpoint}
Using the implicit midpoint method, the Landau-Lifshitz equation is
discretized as
\begin{equation}
  \label{eq:imp_update}
  \frac{\m_i^{j+1} - \m_i^j}{\Delta t} = - \frac{\m_i^j + \m_i^{j+1}}{2} \times \h_i\left(
     \frac{\m^j + \m^{j+1}}{2}\right),
\end{equation}
where
\begin{equation}
  \label{eq:h_term}
  \h_i(\m) := \H_i(\m) + \alpha \m_i \times \H_i(\m).
\end{equation}
  Hence, the values for $\m_i^{j+1}$ are obtained by solving the nonlinear system
  \begin{equation}
    \label{eq:imp_sys}
    \m^{j+1} = \left(\I + \frac{\Delta t}{2} \left[\h \left( \frac{\m^j + \m^{j+1}}{2}\right)\right]_\times\right)^{-1} \left(\I - \frac{\Delta t}{2} \left[\h \left( \frac{\m^j + \m^{j+1}}{2}\right)\right]_\times\right)\m^j,
  \end{equation}
  where $\left[\h\right]_\times$ is the 
  matrix such that the matrix-vector product
  $\left[\h\right]_\times \m^j$ corresponds to taking the
  cross products $\h_i \times \m_i^j$ for all $i \in \{0, ...,  N\}^d$.
  The implicit midpoint method is norm-conserving and results in a
  second order accurate approximation.

  However, when using Newton's method to solve the non-linear system
  \cref{eq:imp_sys}, one has to compute the Jacobian of the
  right-hand side with respect to $\m^{j+1}$, a sparse, but not
  (block)-diagonal, ${3N^d \times 3N^d}$ matrix and then solve the
  corresponding linear system in each iteration, which has a rather
  high computational cost.  In case of the HMM macro model, the
  Jacobian cannot be computed analytically since then $\H_i$ in
  \cref{eq:h_term} is replaced by the averaged quantity
  $\H_\mathrm{avg}(x_i, t_j; \m)$, with a dependence on $\m$ that is
  very complicated. A numerical approximation is highly
  expensive since it means solving $C N^{2d}$ additional micro problems
  per time step.

  \subsubsection*{MPE and MPEA}
  As described in \cite{mpe}, the idea behind the midpoint
  extrapolation method is to approximate
  $\h(\frac{\m^j + \m^{j+1}}{2})$ in \cref{eq:imp_update} using the
  explicit extrapolation formula
  \begin{equation}
    \label{eq:h_mpe}
    \h\left(\frac{\m^j + \m^{j+1}}{2}\right) \approx \h^{j+1/2} := \tfrac{3}{2} \h(\m^j) - \tfrac{1}{2} \h(\m^{j-1}).
  \end{equation}
  The update for $\m_i$ then becomes
  \begin{equation}
  \label{eq:mpe_update}
  \frac{\m_i^{j+1} - \m_i^j}{\Delta t} = - \frac{\m_i^j + \m_i^{j+1}}{2} \times \h_i^{j+1/2}.
\end{equation}
The quantity $\h_i^{j+1/2}$ here is independent of $\m^{j+1}$, which means that the
problem decouples into $N^d$ small $3\times3$ systems.
Since the first term on the right-hand side still contains
$\m_i^{j+1}$, this is considered a semi-implicit method.
Just as the implicit midpoint method, MPE is second order
accurate. We furthermore propose to use third-order accurate
extrapolation,
\begin{equation}
  \label{eq:h_mpea}
  \h^{j+1/2} := \frac{23}{12} \h(\m^j) - \frac{16}{12} \h(\m^{j-1}) + \frac{5}{12} \h(\m^{j-2}),
\end{equation}
in \cref{eq:mpe_update}, which gives MPEA, the adapted MPE method.
As it is based on the implicit midpoint method, MPEA is second order
accurate just like MPE, but has better stability properties for low
damping, as shown in the next section.

As MPE and MPEA are multi-step methods, values for $\m^1$ (and
$\m^2$) are required for startup. These can be obtained using HeunP
or RK4P, as suggested in \cite{mpe}.

\subsubsection*{Comparison of  methods}

In \Cref{fig:ts}, the error with respect to a reference solution
$\m_\mathrm{ref}$ is shown for all the considered (semi-)explicit
methods and two example problems, a 1D problem and a 2D
problem. 
Both example problems are homogenized
problems on a rather coarse spatial discretization grid such as we
might have in the HMM macro problem.
For the 1D problem, the implicit midpoint method is also included for reference.
\begin{figure}[ht!]
  \centering
    \includegraphics[width=.9\textwidth]{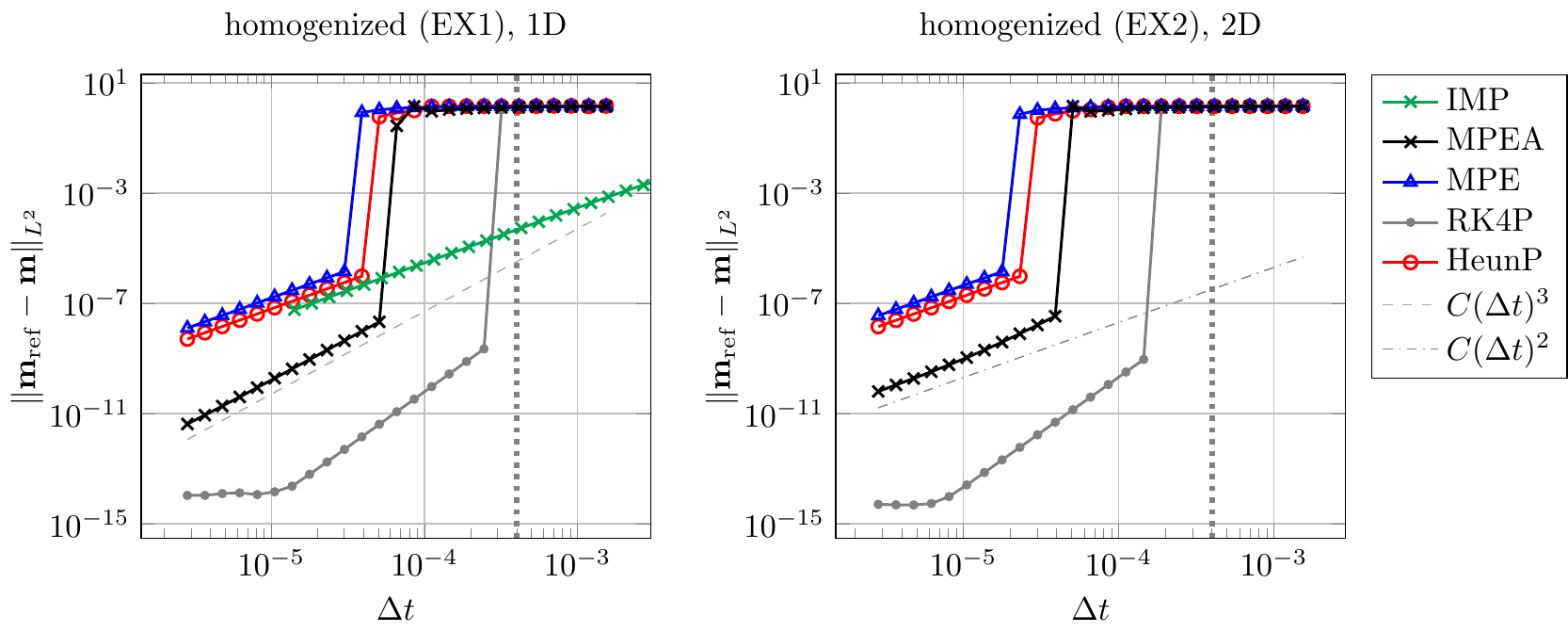}
    \caption{Comparison of $L^2$-error in different time stepping
      methods when varying the time step size $\Delta t$ given a
      fixed spatial discretization with $\Delta x = 1/50$ and
      damping parameter $\alpha = 0.01$.  The dotted vertical line
      corresponds to $\Delta t = (\Delta x)^2$ in each case.}
  \label{fig:ts}
\end{figure}
For small time step sizes, one can observe the expected convergence
rates for HeunP, RK4P and MPE. For MPEA, we observe third order
convergence in the 1D problem, while in the 2D case, we have second
order convergence  for small $\Delta t$. Overall, MPEA results
in lower errors compared to HeunP, MPE and the implicit midpoint
method. RK4P is most accurate.

\subsection{Stability of the time stepping methods}\label{sec:ts_stab}

It is well known that in explicit time stepping methods for the
Landau-Lifshitz equation, the choice of time step size $\Delta t$
given $\Delta x$ in space is severely constrained by numerical
stability, see for example \cite{gspm, llg_lncs}. Note that due to
the norm preservation property of the considered methods, the
solutions cannot grow arbitrarily as unstable solutions typically do
in other applications. However, when taking too large time steps,
explicit time integration will typically result in solutions that
oscillate rapidly and do not represent the intended solution in any
way. Following standard practice in the field, we refer to this
behavior as instability in this section.  This stability limit is
seen clearly for all methods expect IMP in \Cref{fig:ts}.

In numerical experiments, we observe that in order to obtain stable
solutions, the time step size $\Delta t$ has to be chosen
proportional to $\Delta x^2$ for all of the considered methods, both
explicit and semi-explicit.  This is exemplified in \Cref{fig:ts2}
for (EX1) with $\alpha = 0.01$.
\begin{figure}[ht!]
  \centering
    \includegraphics[width=.6\textwidth]{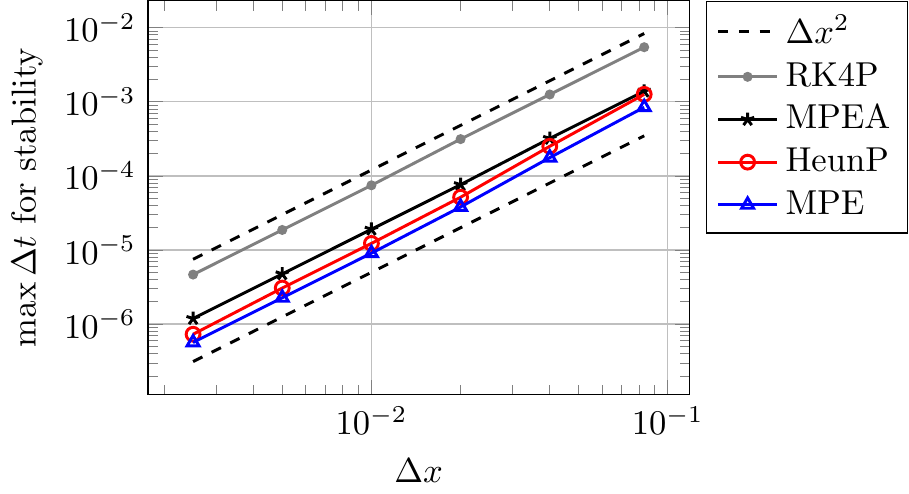}
    \caption{Empirically found maximum value of $\Delta t$ that still results in a
      stable solution for varying values of $\Delta x$ in (EX1), homogenized,  with $\alpha = 0.01$.}
  \label{fig:ts2}
\end{figure}
To get a better understanding of stability, consider a semi-discrete
form of the Landau-Lifshitz equation in one dimension with a
constant material coefficient, equal to one. This can be written as
\begin{equation}
  \label{eq:f_alfa}
  \partial_t \m = \f_\alpha(\m), \qquad \f_\alpha(\m) = - \m \times {D} \m - \alpha \m \times \m \times {D} \m=
  -B(\m)D\m,
\end{equation}
where
$\m\in\Real^{3N}$ contains the vectors $\{\m_i\}$,
$D$ is the discrete Laplacian
and $B(\m)$ is a
block diagonal skew-symmetric matrix with eigenvalues
$\{0,\ +i-\alpha,\ -i-\alpha\}$ that comes from the cross products.
If $\{\m_i\}$ samples a smooth function, the
Jacobian of $\f_\alpha(\m)$ can be approximated as
\[
\nabla_\m \f_\alpha(\m)=-\nabla_\m B(\m)D\m
\approx -B(\m)D.
\]
Still assuming smoothness, one can subsequently deduce \cite{thesis}
that the eigenvalues of the Jacobian are approximately given as the
eigenvalues $\omega$ of $-D$ multiplied by the eigenvalues of
$B(\m)$, namely
\[
\lambda_+\approx (i-\alpha)\omega, \qquad
\lambda_-\approx (-i-\alpha)\omega, \qquad
\lambda_0\approx 0.
\]
The eigenvalues of $-D$, the negative discrete Laplacian, are real,
positive and bounded by $O(\Delta x^{-2})$. Consequently, the
eigenvalues of the Jacobian $\nabla_\m \f_\alpha(\m)$ will lie along
the lines $s(\pm i-\alpha)$ for real $s\in[0,O(\Delta x^{-2})$] in
the complex plane.  This is illustrated in \Cref{fig:evals_J} where
we have plotted the eigenvalues of $\nabla_\m \f_\alpha(\m)$, scaled
by $\Delta x^2$, for several values of $\alpha$.  One can observe
that given $\alpha = 0$, the eigenvalues are purely imaginary. As
$\alpha$ increases, the real parts of the eigenvalues decrease
correspondingly.

For the Landau-Lifshitz equation
\cref{eq:prob} with a material coefficient as well as the
homogenized equation \cref{eq:hom}, the eigenvalues of the corresponding Jacobians
get a different scaling based on the material coefficient but their
general behavior is not affected.
We hence conjecture that it is necessary that
\begin{align}\label{eq:Cstab}
  \frac{\Delta t}{(\Delta x)^2} \le C_{\mathrm{stab}, \alpha},
\end{align}
where $C_{\mathrm{stab}, \alpha}$ is a constant depending on the
chosen integrator, the damping parameter $\alpha$ and the material
coefficient.  Based on several numerical examples, we observe for
the latter dependence that
\[C_{\mathrm{stab}, \alpha} \lesssim C_\mathrm{\alpha}
  \begin{cases}
    (\max_{y \in Y} |a(y)|)^{-1}, & \text{original problem},  \\
    (\max_{i,j} |A^H_{ij}| )^{-1}, & \text{homogenized problem},
  \end{cases}
\]
where $C_\mathrm{\alpha}$ denotes further dependence on $\alpha$ and the integrator.

\subsubsection*{Stability regions of related methods}

In order to better understand the stability behavior of the
considered time integrators, it is beneficial to study the
stability regions of some well-known, related methods.  For HeunP
and RK4P, we regard the corresponding integrators without
projection. We observe a very similar stability behavior when using Heun
and RK4 to solve the problems considered in \Cref{fig:ts,fig:ts2}.

To get some intuition about MPE(A), we start by considering the problem
\begin{align}\label{eq:LL_lin}
  \partial_t \m = - \n \times \H(\m) - \alpha \n \times \m \times \H(\m)
  = - \n \times \h(\m)
  ,
\end{align}
where $\n$ is a given vector function, constant in time, with
$|\n| = 1$. This corresponds to replacing the first $\m$ in each term on the
right-hand side in \cref{eq:prob} by a constant approximation.
For this problem, time stepping according to the MPE
update \cref{eq:mpe_update} results in
\begin{equation}
  \label{eq:mpe_update_ab}
  \frac{\m_i^{j+1} - \m_i^j}{\Delta t} = - \n_i \times \h_i^{j+1/2}
  = - \frac{3}{2} \left(\n_i \times \h_i(\m^j)\right) + \frac{1}{2} \left(\n_i \times \h_i(\m^{j-1})\right).
\end{equation}
This is the same update scheme as one gets when applying the
Adams-Bashforth 2 (AB2) method to \cref{eq:LL_lin}. In the same way,
MPEA and AB3 are connected. Furthermore, note that the term that was
replaced by $\n$ in \cref{eq:prob} to obtain \cref{eq:LL_lin} is the
one that is treated implicitly in the semi-implicit methods MPE(A).
We hence expect that studying the stability of AB2(3) can give an
indication of what to expect for MPE(A). This is backed up by the
fact that the stability properties observed for MPE(A) in
\Cref{fig:ts,fig:ts2} are closely matched when using AB2(3) to solve
the respective problems.

When comparing the stability regions of RK4, Heun, AB2 and AB3 as
shown in \Cref{fig:stab_reg}, one clearly sees that the Runge-Kutta
methods have larger stability regions than the multi-step
methods. RK4's stability region is largest and contains part of the
imaginary axis, while the one for Heun is only close to the
imaginary axis in a shorter interval.
\begin{figure}[h!]
    \begin{subfigure}[b]{0.45\textwidth}
    \centering
    \includegraphics[width=\textwidth]{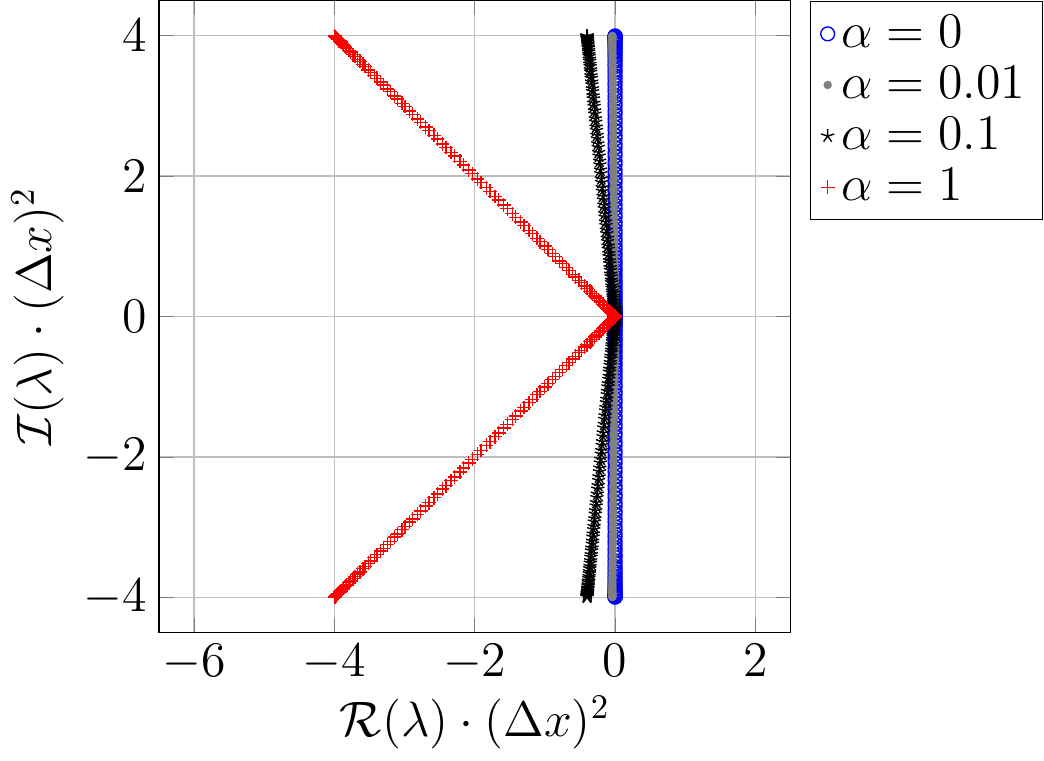}
    \caption{Eigenvalues of the Jacobian
      $\nabla \f_\alpha(\m)$ for several values of $\alpha$}
     \label{fig:evals_J}
   \end{subfigure}
   ~
  \begin{subfigure}[b]{0.45\textwidth}
    \centering
    \includegraphics[width=\textwidth]{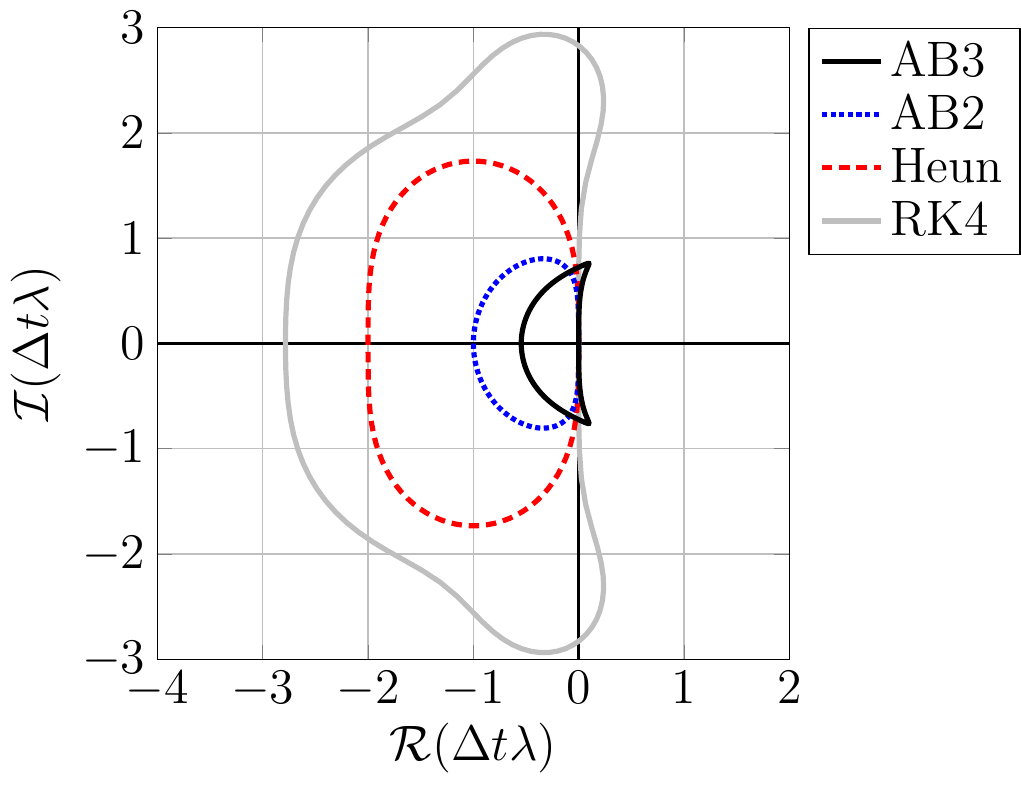}
    \caption{Stability regions of RK2 (Heun), RK4 as well as AB2 and AB3.}
    \label{fig:stab_reg}
  \end{subfigure}
  \caption{Eigenvalues of Jacobian $\nabla \f_\alpha(\m)$ and
    stability regions of methods related to the considered time integrators.}
\end{figure}
AB2 and AB3 have stability regions with a similar extent in the
imaginary direction, but while AB3's contains part of the imaginary
axis, AB2's does not. On the other hand, the stability region of AB2 is
wider in the real direction.

Consider now again the example problems shown in \Cref{fig:ts} where
$\alpha=0.01$, which implies that the eigenvalues of the
corresponding Jacobian are rather close to the imaginary axis. Based
on \Cref{fig:stab_reg} we therefore expect the methods with related
stability areas which include parts of the imaginary axis, RK4P and
MPEA, to require fewer time steps than HeunP and MPE. This matches
with the observed stability behavior in \Cref{fig:ts}.

\subsubsection*{Influence of $\alpha$}

To further investigate the influence of $\alpha$ on
$C_\mathrm{stab, \alpha}$, this factor is shown in
\Cref{fig:stab_lim} for varying $\alpha$, both for the considered
methods HeunP, RK4P, MPE and MPEA and the discussed related methods,
Heun, RK4, AB2 and AB3.  The behavior of $C_\mathrm{stab, \alpha}$
is almost the same for the actual and the related methods.

As expected, we observe that $C_\mathrm{stab, \alpha}$ for low
$\alpha$-values is constant for  MPEA and RK4P, with related
stability regions that include the imaginary axis, while
for HeunP and MPE, lower $\alpha$ results in lower $C_\mathrm{stab, \alpha}$.
When increasing $\alpha$, the eigenvalues of $\bnabla_m \f_\alpha$
as defined in \cref{eq:f_alfa} get larger real parts and the
stability regions' extent in the real direction becomes more
important.  For RK4P and MPEA, this means that for
$\alpha \gtrsim 0.2$, $C_\mathrm{stab, \alpha}$ decreases as
$\alpha$ increases.  For HeunP and MPE, the highest
$C_\mathrm{stab, \alpha}$ is obtained around $\alpha = 0.5$.  For
higher $\alpha$, the required $C_\mathrm{stab, \alpha}$ decreases as
$\alpha$ increases.
Overall, MPEA requires the lowest $C_{\mathrm{stab}, \alpha}$ for
high $\alpha$, which agrees with the fact that the related stability
region is shortest in the real direction. The highest
$C_{\mathrm{stab}, \alpha}$ is still the one for RK4P, in accordance
with the stability region considerations.

\begin{figure}[h!]
  \centering
    \includegraphics[width=.9\textwidth]{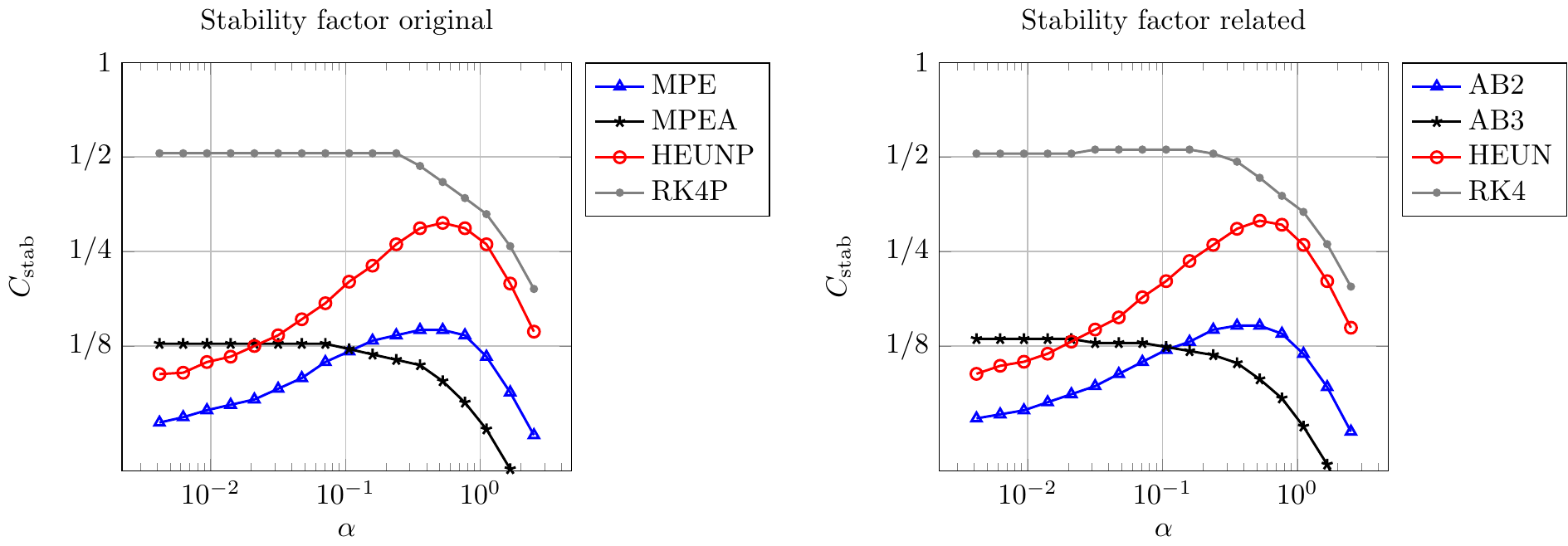}
    \caption{Dependence of $C_{\mathrm{stab}, \alpha}$ on $\alpha$, for actual and related time integrators. Based on the homogenized solution to
      (EX2). }
  \label{fig:stab_lim}
\end{figure}
However, HeunP is a two-stage method and each RK4P step consists of
four stages, while MPE and MPEA are multi-step methods that only
require computation of one new stage per time step.  In general,
each step of HeunP/RK4P thus has roughly two/four times the
computational cost as a MPE(A) step. To take this into account, we
compare the total number of computations for each method by
considering the factor $s/C_{\mathrm{stab}, \alpha}$, where $s$
denotes the number of stages in the method. This is shown in
\Cref{fig:stab_cost}.
\begin{figure}[h!]
  \centering
    \includegraphics[width=.55\textwidth]{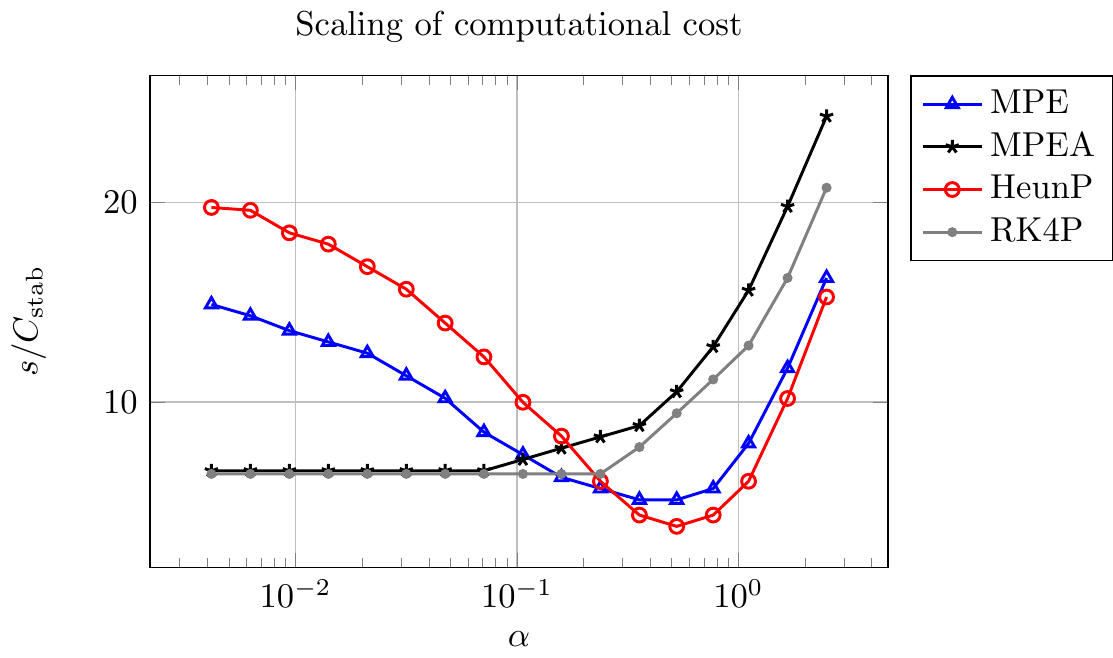}
    \caption{Left: Dependence of $C_{\mathrm{stab}, \alpha}$ on $\alpha$, based on the homogenized solution to
      (EX2). Right: corresponding scaling of computational cost.}
  \label{fig:stab_cost}
\end{figure}

{
We hence draw the following conclusion.
\begin{itemize}
\item For $\alpha < 0.1$, RK4 and MPEA result in approximately the
  same computational cost, independent of the specific value of $\alpha$, while MPE and
  HeunP require significantly more computations.
\item For high (artificial) damping, the situation changes and HeunP
  has the lowest computational cost of the considered time
  integrators.
\end{itemize}
}

\subsection{Macro time stepping}\label{sec:macro_ts}

On the HMM macro scale, the given spatial discretization is in
general rather coarse, containing only relatively few grid points,
and we are interested in longer final times. Therefore, an implicit
method such as IMP might seem suitable here. However, the resulting
computational cost is higher than with an explicit method since we
cannot compute the required Jacobian analytically as discussed in
\Cref{sec:description}.

When considering which (semi)-explicit method is most suitable, we
have to consider the value of the damping constant $\alpha$.
According to for example \cite{recording_materials, mayergoyz2009},
the value of $\alpha$ is less than 0.1 or even 0.01 for typical
metallic materials such as Fe, Co and Ni, and could be one or two
orders smaller for ferromagnetic oxides or garnets.  On the macro
scale, we hence typically reckon with $\alpha$ between $10^{-1}$ and
$10^{-4}$. This also matches the $\alpha$-values typically used in
the literature, see for instance \cite{gspm, garcia2003improved,
  garcia_review} and \cite{mpe}. For this range of $\alpha$, we
conclude based on the discussion in \Cref{sec:ts_stab} that RK4P and
MPEA are preferred for time integration on the macro scale. These
are also the methods which give the most accurate solutions as, for
example, shown in \Cref{fig:ts}.

For the overall error on the macro scale in the periodic case, we expect that
\begin{align}\label{eq:e_macro}
  \|\M^\varepsilon - \M_0\| \le C \left(\varepsilon + (\Delta t)^{k} + (\Delta X)^\ell + e_\mathrm{HMM}\right)
  \le C \left( \varepsilon + (\Delta X)^{\min(2k, \ell)} +  e_\mathrm{HMM}\right),
\end{align}
where the factor $\varepsilon$ follows from \Cref{thm:paper1} and
$k$ is the order of accuracy of the time integrator. Moreover, $\ell$
is the order of accuracy of the spatial approximation to the
effective field on the macro scale
and $e_\mathrm{HMM}$ is an additional error due to the fact that we
approximate this effective field by $\H_\mathrm{avg}$ in the upscaling
procedure.

Because of the time step restriction required for stability
in explicit and semi-explicit methods, \cref{eq:Cstab}, it is
desirable to have relatively large $\Delta x$ to reduce computational cost. We
therefore propose to use a higher order method in space, ideally $\ell = 2k$. 

We hence have two possible choices. One can either use the fourth
order accurate RK4P in combination with an eighth order
approximation in space to get a macro scheme with very high order of
accuracy regarding space and time discretization. However, to get
the full effect of this, $\H_\mathrm{avg}$ has to be approximated
very precisely such that also $e_\mathrm{HMM}$ is low, see also
\Cref{sec:overall}, which in turn can result in rather high computational cost.
Alternatively, one can apply the second order MPEA and a spatial approximation such that $\ell = 4$
which results in fourth order accuracy of the space and time
discretization. Since MPEA also has the advantage that it is a geometric
integrator and norm preserving without projection, this is what we
propose to use. 

\subsection{Micro time stepping}

When considering the micro problem, there are two important
differences compared to the macro problem.  First, the fast
oscillations in the solution $\m^\varepsilon$ are on different
scales in time and space. The time scale we are interested in is
$\mathcal{O}(\varepsilon^2)$ while the spatial scale is
$\mathcal{O}(\varepsilon)$.  Therefore a time step size proportional
to $(\Delta x)^2$ is suitable to obtain a proper resolution of the
fast oscillations in time. Second, as discussed
\Cref{sec:upscaling}, we can choose the damping parameter $\alpha$
for the micro problem to optimize convergence of the upscaling
errors.  As shown in \Cref{sec:micro}, it is typically advantageous
to use artificial damping and set $\alpha$ close to one, considerably higher than
in the macro problem.

The order of accuracy of the time integrator is not an important
factor, since already for a second order accurate
integrator, 
the time integration error usually is significantly lower than the
space discretization error due to the given time step restriction.
As the micro problem is posed on a relatively
short time interval and the solution then is averaged in the
upscaling process, inherent norm preservation that geometric
integrators have is not an important factor here either. The
considerations in \Cref{sec:ts_stab} thus imply that the
optimal strategy with respect to computational cost is to use
HeunP for time integration when $\alpha > 0.2$ is chosen in the micro problem. 

\section{Micro problem setup}\label{sec:micro}

In this section, we investigate how different aspects of the micro
problem influence the upscaling error as well as the overall
macro solution. In particular, the choice of initial data
and the size of the computational and averaging domain in space and
time are important.  Consider the periodic case,
$a^\varepsilon(x) = a(x/\varepsilon)$. Then, it holds for the error
in the HMM approximation to the effective field that
 \begin{align}\label{eq:E_approx}
   E_\mathrm{approx}&:= \left|\H_\mathrm{avg} - \bnabla \cdot (\bnabla \M_0(x_i, t_j) \A^H)\right| \nonumber
   \\ &\le
        \left|\H_\mathrm{avg} - \bnabla \cdot (\bnabla \m_\mathrm{init}(0) \A^H)\right|
        + \left|\bnabla \cdot (\bnabla \m_\mathrm{init}(0) \A^H)
        - \bnabla \cdot (\bnabla \M_0(x_i, t_j) \A^H)\right| \nonumber
        \\&=: E_\mathrm{avg} + E_\mathrm{disc}.
 \end{align}
 The discretization error $E_\mathrm{disc}$ is determined by the
 choice of initial data $\m_\mathrm{init}$ to the micro problem and
 is analyzed in the next section, \Cref{sec:initial}.
 Given that we have initial data with $|\m_\mathrm{init}| = 1$ and
 an exact solution to the micro problem on the whole domain, the
 averaging error $E_\mathrm{avg}$ can be bounded using
 \Cref{thm:paper2}.  When solving the micro problem numerically and
 only on a subdomain $[-\mu', \mu']^d$, additional errors are
 introduced. We can split $E_\mathrm{avg}$ as
 \begin{align}\label{eq:E_avg}
   E_\mathrm{avg} = E_\varepsilon + E_\mu + E_\eta + E_{\mu'} + E_\mathrm{num},
 \end{align}
 where $E_\varepsilon, E_\mu$ and $E_\eta$ are as in
 \Cref{thm:paper2}. They depend on $\varepsilon$ and the parameters
 $\mu$ and $\eta$, which determine the size of the micro problem
 averaging domains in space and time. Moreover, the choice of
 averaging kernels, $K$ and $K^0$ influences these errors. How to
 specifically choose these parameters is discussed in
 \Cref{sec:micro_size}. The term $E_{\mu'}$ comprises errors due to
 the micro problem boundary conditions, as explained in
 \Cref{sec:bc}, and $E_\mathrm{num}$ errors due to the numerical
 discretization of the micro problem, which is done using a standard
 second order finite difference approximation to
 $\bnabla \cdot (a^\varepsilon \bnabla \m^\varepsilon)$ and HeunP
 for time integration.  Throughout the following sections, we assume
 that $E_\mathrm{num}$ is small compared to the other error terms
 and can be neglected.

\subsection{Initial data}\label{sec:initial}
We first consider how to choose the initial data $\m_\mathrm{init}$
to a micro problem based on the current macro state, obtained according to
\cref{eq:macro_disc}. 
We here suppose that the current given discrete magnetization values
match with the (exact) macro solution at time $t = t_j$,
$\M_i = \M(x_i, t_j)$.

The initial data $\m_\mathrm{init}$ for the micro problem
should satisfy two conditions.
\begin{enumerate}
\item It should be normalized, $|\m_\mathrm{init}(x)| = 1$ for all
  $x \in [-\mu', \mu']^d$, to satisfy the conditions necessary for \Cref{thm:paper2}, which we use to bound $E_\mathrm{avg}$.
\item The initial data should be consistent with the current macro solution
  in the sense that given a multi-index $\beta$ with $|\beta| \le 2$,
  \begin{align}\label{eq:consistency}
    \left|\partial_x^\beta \m_\mathrm{init}(0) - \partial_x^\beta \M(x_k, t_j)\right| =  \mathcal{O}((\Delta X)^\ell).
  \end{align}
  Then the discretization error in \cref{eq:E_approx} is
  $E_\mathrm{disc} = \mathcal{O}((\Delta X)^\ell)$.  As described in
  \cref{eq:e_macro} in \Cref{sec:macro_ts}, when using a $k$th order
  explicit time stepping method, it is ideal in terms of order of
  accuracy to choose $\ell = 2k$.

\end{enumerate}
%
In order to get initial data satisfying the requirements, an
approach based on polynomial interpolation is applied.  We use
$p^{[n]}(x)$ to denote an interpolating polynomial of order $n$, and
let $\P_{[n]}(x) = \left[p_1^{[n]}, p_2^{[n]}, p_3^{[n]}\right]^T$, a vector
containing an independent polynomial for each component in $\M$,
such that
\[\P_{[n]}(x_i) = \M_{i}, \quad i = 0, \ldots,  n.\]
When $d > 1$, we apply one-dimensional interpolation
in one space dimension after the other. For matters of simplicity,
we regard a 1D problem in the following analytical error
estimates. Due to the tensor product extension, the considerations
generalize directly to higher dimensions.

Without loss of generality, we henceforth assume that we want to
find initial data for the micro problem associated with the macro
grid point at location $x_{k}$ based on $2k$-th order polynomial
interpolation. This implies that the macro grid points involved in
the process are $x_j$, $j = 0, ..., 2k$.

According to standard theory, it holds for the interpolation errors that
given $0 \le i \le 2k$ and $\M \in C^{(2k+1)}([x_0, x_{2k}])$,
\begin{align}\label{eq:interpol_err}
  \sup_{x \in [x_0, x_{2k}]} |\P_{[2k]}^{(i)}(x) - \M^{(i)}(x, t_j)| 
  \le C (\Delta X)^{2k+1-i}.
\end{align}
Furthermore, it is well known that given a $2k$-th degree interpolating polynomial, it holds that
\begin{align*}
  \P_{[2k]}'(x_k) = D_{[2k]} \M(x_k, t_j), \qquad \P_{[2k]}''(x_k) = D_{[2k]}^2 \M({x_k}, t_j),
\end{align*}
where $D_{[2k]}$ and $D^2_{[2k]}$ denote the $2k$-th order standard central
finite difference approximations to the first and second
derivative, see for example \cite{Peiro2005}. As a direct consequence, we have
in the grid point $x_k$,
\begin{align}\label{eq:der_interpol_err}
  |\M'(x_k, t_j) - \P_{[2k]}'(x_k)| \le C (\Delta X)^{2k}, \qquad |\M''(x_k, t_j) - \P_{[2k]}''(x_k)| \le C (\Delta X)^{2k}.
\end{align}
Note that this gives a better bound for the error in the second
derivative in the point $x_k$ than \cref{eq:interpol_err}, valid on
the whole interval $[x_0, x_{2k}]$. The bounds in\Cref{eq:der_interpol_err} show
that we have the required consistency, \cref{eq:consistency}, between macro and micro
derivatives when directly using $\P_k(x)$ to obtain the initial data
for the micro problem. However, the disadvantage of this approach is
that the polynomial vector $\P_k$ is not normalized. For the
deviation of its length from one, it holds by \cref{eq:interpol_err}
that
\[\left||\P_{[2k]}| - 1\right| = \left| |\P_{[2k]}| - |\M|\right| \le |\P_{[2k]} - \M| \le C (\Delta X)^{2k+1}.\]

Consider instead a normalized function $\Y(x)$ for which
$|\Y(x)| = 1$. Then the derivative $\Y'$ is orthogonal to $\Y$ as it
holds that
\[\Y'(x)\cdot \Y(x) = \frac{1}{2} \frac{d}{dx} |\Y(x)|^2 = 0.\]
In particular, this shows that $\M' \cdot \M = 0$. However, in general,
\[\M(x_k, t_j) \cdot D_{[2k]} \M(x_k, t_j) = \P_{[2k]}(x_k) \cdot \P_{[2k]}'(x_k) \neq 0.\]
Hence there is no normalized interpolating
function $\Y$ such that $\Y'(x_\mathrm{k})$ becomes a standard linear $2k$-th
order central difference approximation, $D_{[2k]}$.
In the following, we consider the normalized interpolating function $\Q_{[n]}(x)$ defined as
\begin{equation}
  \label{eq:Q_k}
  \Q_{[n]}(x) := \P_{[n]}(x) / |\P_{[n]}(x)|\,,
\end{equation}
and show that it satisfies the consistency requirement \cref{eq:consistency}.

\begin{lemma}
  In one space dimension, the normalized function $\Q_{[2k]}$
  satisfies \cref{eq:consistency} with $\ell = 2k$.
\end{lemma}

\begin{proof}
 As $|\P_{[2k]}(x_i)| = |\M_i| = 1$ in the grid points $x_i$,
 where $i = 0, ..., 2k$, it follows directly that $\Q_{[2k]}(x_i) = \M_i$ for
 $i = 0, ... , 2k$, the normalized function still interpolates the
 given points.

For the first derivative of $\Q_{[2k]}$, it holds that
\begin{equation*}
  \begin{split}
  \Q_{[2k]}' &= \frac{d}{dx} \left(\P_{[2k]} / |\P_{[2k]}|\right)
= \P_{[2k]}' / |\P_{[2k]}| -  (\P_{[2k]}^T \P_{[2k]}') \P_{[2k]} /|\P_{[2k]}|^3 \,,
  \end{split}
\end{equation*}
which, together with the fact that $\P_{[2k]}(x_i) = \M_i$ for $i = 0, .., 2k$, implies that
\begin{equation}\label{eq:Q_der}
  \Q_{[2k]}'(x_i) 
= (\I - \M_i \M_i^T) \P_{[2k]}'(x_i).
\end{equation}
In particular, it holds due to orthogonality that
\begin{align*}
  \left|\Q_{[2k]}'(x_i) - \M'(x_i)\right|
  &= \left| (\I - \M_i \M_i^T) \left(\P_{[2k]}'(x_i) - \M'(x_i)\right)\right| \\
    &\le C |\P_{[2k]}'(x_i) - \M'(x_i)| \le C (\Delta X)^{2k},
\end{align*}
  hence $\Q_{[2k]}'(x_k)$ is a $2k$-th order approximation to the derivative of $\M$ in
  $x_k$.
  For the second derivative of $\Q_{[n]}$, it holds in general that
  \begin{equation}\label{eq:Q_der2}
  \Q_{[n]}'' = \frac{\P_{[n]}''}{|\P_{[n]}|} - 2 \frac{(\P_{[n]}\cdot\P_{[n]}')\P_{[n]}'}{|\P_{[n]}|^3} + 3 \frac{(\P_{[n]} \cdot \P_{[n]}')^2 \P_{[n]}}{|\P_{[n]}|^5} - \frac{(|\P_{[n]}'|^2 + \P_{[n]}'' \cdot \P_{[n]}) \P_{[n]}}{|\P_{[n]}|^3},
\end{equation}
where we can rewrite
\[ |\P_{[n]}'|^2 + \P_{[n]}'' \cdot \P_{[n]}  = (\P_{[n]}- \M) \cdot \P_{[n]}'' + (\P_{[n]}' - \M') \cdot (\M' + \P_{[n]}') + (\P_{[n]}'' - \M'') \cdot \M,\]
For a  $2k$-th order interpolating polynomial, we hence have in the grid point $x_k$ that
\begin{align*}
  &\left| \left(|\P_{[2k]}'(x_k)|^2 + \P_{[2k]}''(x_k) \cdot \P_{[2k]}(x_k)\right) \P_{2k}(x_k)\right| \\
  &\hspace{1cm}\le \left|(\P_{[2k]}'(x_k) - \M'(x_k)) \cdot (\M'(x_k) + \P_{[2k]}'(x_k)) + (\P_{[2k]}''(x_k) - \M''(x_k)) \cdot \M(x_k)\right| \\
    &\hspace{1cm} \le C \left(\left|(\P_{[2k]}'(x_k) - \M'(x_k)) \right| + \left|\P_{[2k]}''(x_k) - \M''(x_k))\right|\right) \le C (\Delta X)^{2k},
\end{align*}
where we used \cref{eq:der_interpol_err} in the last step.
Moreover, it holds due to orthogonality and \cref{eq:der_interpol_err} that
\[ |\P_{[2k]}(x_k) \cdot \P_{[2k]}'(x_k)| = |\M_k \cdot (\P_{[2k]}'(x_k) - \M'(x_k))|
  \le |\P_{[2k]}'(x_k) - \M'(x_k)| \le C (\Delta X)^{2k}.\]
It therefore follows that
\begin{align*}
  |\Q_{[2k]}''(x_k) - \P_{[2k]}''(x_k)|
  &\le 2\left|(\P_{[2k]}(x_k) \cdot \P_{[2k]}'(x_k)) \P_{[2k]}'(x_k) \right|
    \\ &+ 3  \left|\left(\P_{[2k]}(x_k) \cdot \P_{[2k]}'(x_k)\right)^2 \P_{[2k]}(x_k) \right|
    \\ &+ \left| \left(|\P_{[2k]}'(x_k)|^2 + \P_{[2k]}''(x_k) \cdot \P_{[2k]}(x_k)\right) \P_{2k}(x_k)\right|\le C (\Delta X)^{2k},
\end{align*}
which by \cref{eq:der_interpol_err} implies that $\Q_k''(x_k)$ is a
$2k$-th order approximation to $\M''$ (but not a standard linear
central finite difference approximation).
\end{proof}

We can conclude that when using either $\P_{[2k]}$ or $\Q_{[2k]}$ to
obtain initial data $\m_\mathrm{init}$ for the HMM micro problem,
then $\partial_{xx} \m_\mathrm{init}(0)$ is a $2k$-th order
approximation to $\partial_{xx} \M(x_k, t_j)$, where $x_k$ is the
grid point in the middle of the interpolation stencil. Similarly, in
two space dimensions, $\partial_{xy} \m_\mathrm{init}(0)$ and
$\partial_{yy} \m_\mathrm{init}(0)$ are $2k$-th order approximations
to $\partial_{xy} \M(x_k, t_j)$ and $\partial_{yy} \M(x_k, t_j)$.
Thus, the discretization error, which corresponds to the interpolation error, is
\begin{align}\label{eq:E_interpol}
  E_\mathrm{disc} = \bnabla \cdot (\bnabla \m_\mathrm{init}(0) \A^H) - \bnabla \cdot (\bnabla \M(x_k, t_j) \A^H) = C (\Delta
  X)^{2k}.
\end{align}
Both $\P_{[2k]}$ and $\Q_{[2k]}$ hence satisfy the
consistency requirement \cref{eq:consistency} for the initial data. However, only $\Q_{[2k]}$
 is normalized, therefore this is
what we choose subsequently. Typically, the difference between the
approximations is only rather small, though, as shown in the
following numerical example.

\subsubsection*{Numerical example}
As an example, consider the initial data for a micro problem on a 2D
domain of size $10 \varepsilon$ in each space dimension, obtained by
interpolation from the macro initial data of (EX2) and
(EX3).

We first investigate the maximal deviation of the length of
$\m_\mathrm{init}$ from one when using $\P_{[2]}$ and $\P_{[4]}$ to obtain the
initial data. In \Cref{fig:norm_dev}, this error is shown for
varying $\Delta X$ and different values of $\varepsilon$.
\begin{figure}[ht]
  \centering
   \includegraphics[width=.9\textwidth]{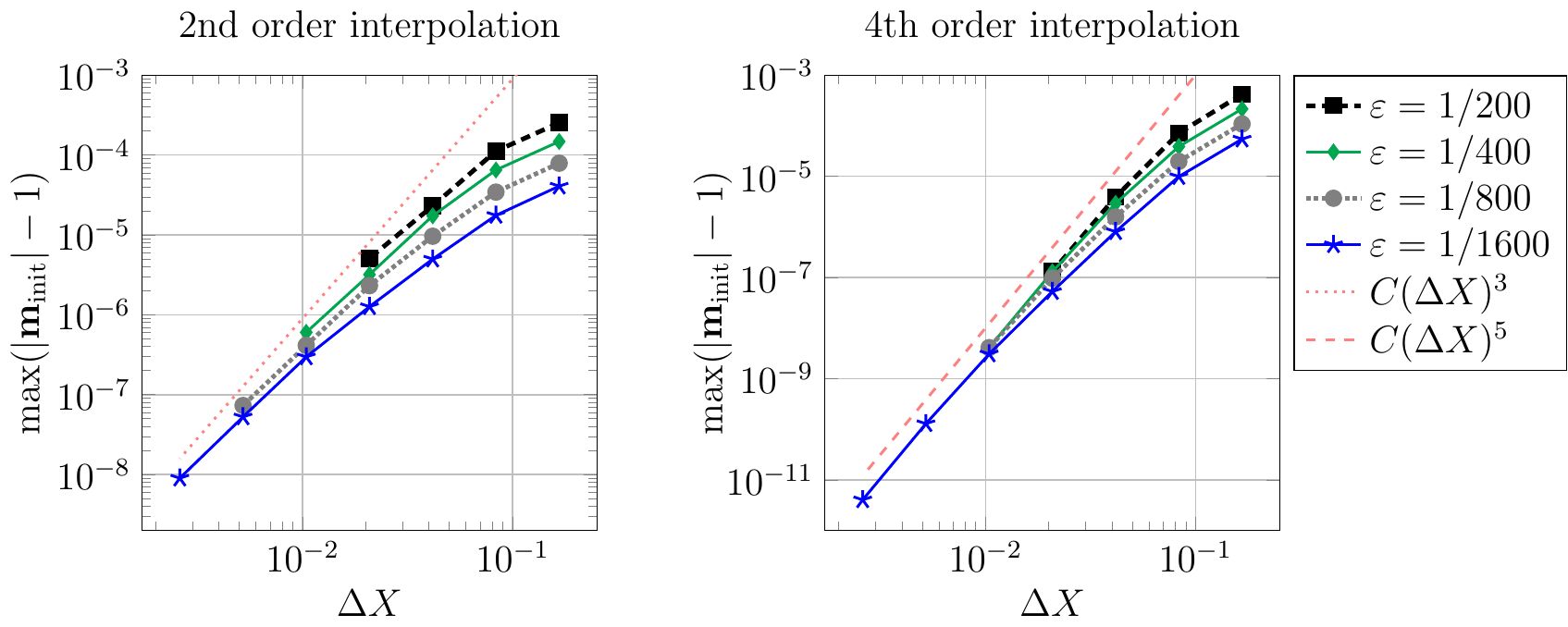}
   \caption{Maximum norm deviation in polynomial
     interpolation initial data $\P_{[2]}$ (left) and $\P_{[4]}$ (right) from one given a micro
     domain size of $10\varepsilon$, for several values of $\varepsilon$. 2D problem
     with macro initial data as in (EX3) and (EX2).
     Only values where $2k \Delta X > 10\varepsilon$ are plotted
     to avoid extrapolation.
   }
  \label{fig:norm_dev}
\end{figure}

One can observe that the deviation decreases as the macro step size
$\Delta X$ decreases. Moreover, especially for high $\Delta X$ values, smaller
$\varepsilon$ result in smaller deviations. This is due to the fact
that a smaller $\varepsilon$ corresponds to a smaller micro domain,
around $x_k$. The maximum possible norm deviation is only
attained further away from $x_k$. In the limit, as
$\varepsilon \to 0$, the norm deviation vanishes.


Next, we examine the difference between $\P_{[2k]}$ and $\Q_{[2k]}$
and the order of accuracy of the resulting approximations to
$\bnabla \cdot (\bnabla \M \A^H)$.  In the left subplot in
\Cref{fig:interpol}, the error $E_\mathrm{disc}$ is shown for
several values of $\Delta X$ and for second and fourth order
interpolation. The expected convergence rates of
$(\Delta X)^{2k}$ can be observed for both approximations, based on
$\P_\mathrm{[2k]}$ and $\Q_\mathrm{[2k]}$.
\begin{figure}[ht!]
  \centering
   \includegraphics[width=.75\textwidth]{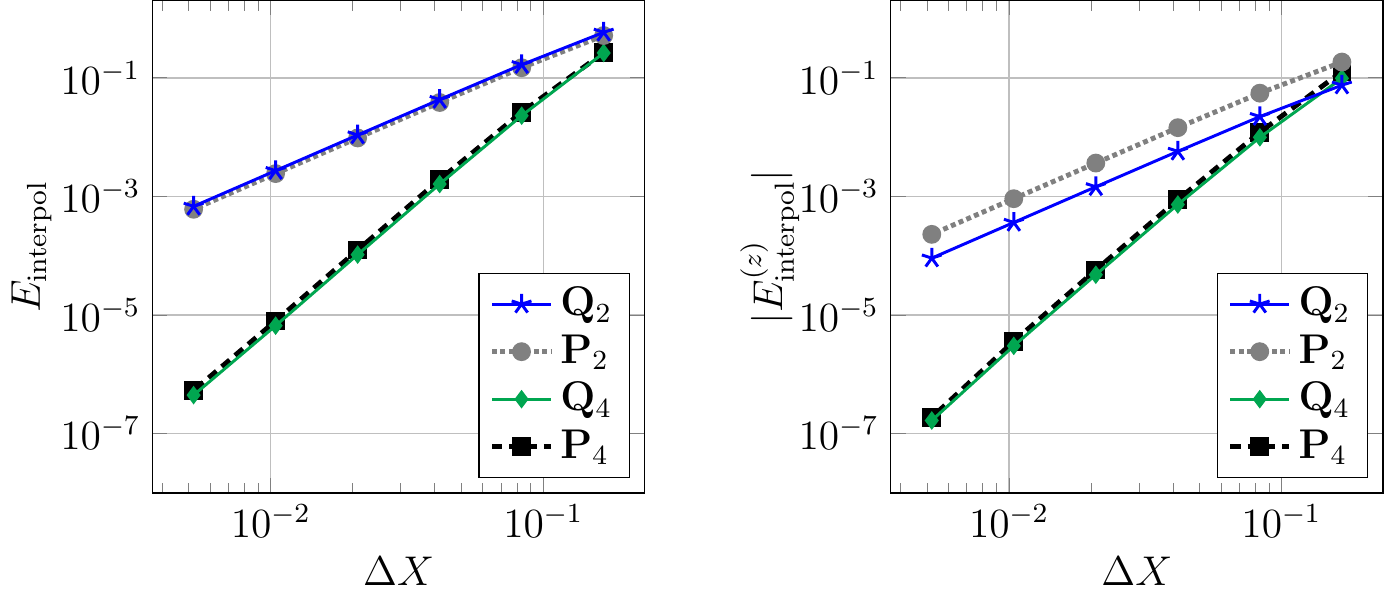}
   \caption{Interpolation error $E_\mathrm{disc}$ with
     and without normalization, i.e. using $\Q_{[2k]}$ and $\P_{[2k]}$, when varying
     macro grid spacing $\Delta X$. $2$nd and $4$th order interpolation. Left:
     norm of the error between approximated and actual effective
     field for (EX2). Right: $z$-component only.}
  \label{fig:interpol}
\end{figure}
In the right subplot, only the difference between the
$z$-components of $\bnabla \cdot (\bnabla \m_\mathrm{init} \A^H)$ and
$\bnabla \cdot (\bnabla \M \A^H)$ is considered to emphasize the fact
that while $\P_{[2k]}$ and $\Q_{[2k]}$ result in approximations of the same order
of accuracy, they do not give the same approximation.

\subsection{Boundary conditions}\label{sec:bc}
In this section, the issue of boundary conditions for the HMM micro
problem is discussed. In the case of a periodic material
coefficient, as considered for the estimate in \Cref{thm:paper2},
the micro problem would ideally be solved on the whole domain
$\Omega$ with periodic boundary conditions, even though the
resulting solution is only averaged over a small domain
$[-\mu, \mu]^d$.  This is not a reasonable choice in practice, since
the related computational cost is too high. We therefore have to
restrict the size of the computational domain for the micro problem
and complete it with boundary conditions. Every choice of boundary
conditions introduces some error inside the domain in comparison to
the whole domain solution since it is not possible to exactly
match both ``incoming'' and ``outgoing'' dynamics. In
\Cref{fig:edn}, the effect of boundary conditions in comparison to
the solution on a much larger domain is illustrated for one example
1D microproblem with the setup (EX1). The solution in the micro
domain as well as the errors due to two kinds of boundary conditions
are plotted: assuming periodicity of
$\m^\varepsilon - \m_\mathrm{init}$ (middle) and homogeneous
Dirichlet boundary conditions for
$\m^\varepsilon - \m_\mathrm{init}$ (right).
\begin{figure}[h]
  \centering
   \includegraphics[width=.31\textwidth, trim=40 30 10 30]{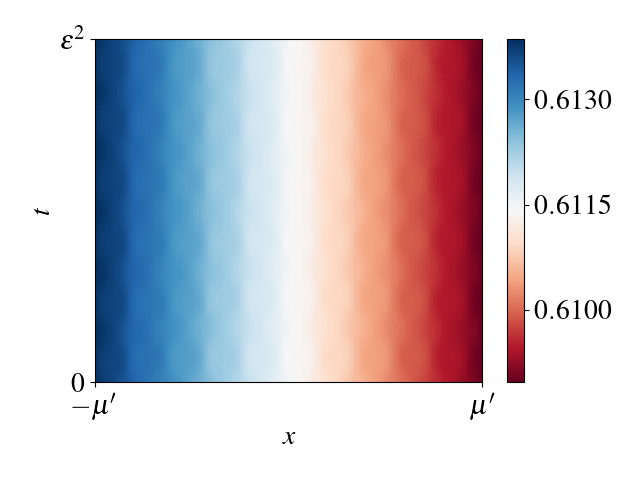}
  ~
  \includegraphics[width=.31\textwidth, trim=40 30 10 30]{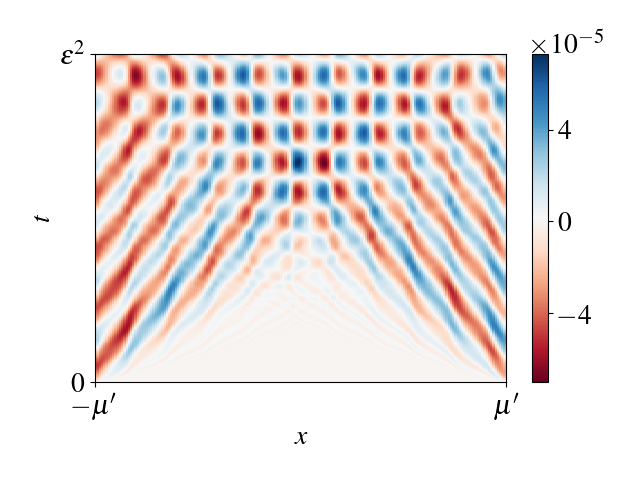}
  ~
  \includegraphics[width=.31\textwidth, trim=40 30 10 30]{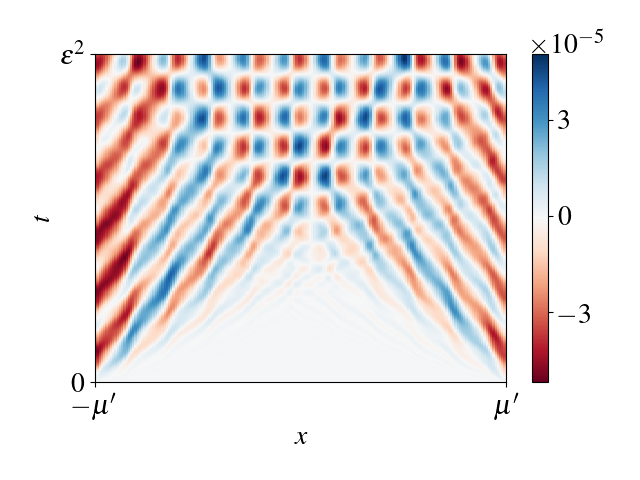}

  \caption{Example (EX1) with $\alpha = 0.01$, comparison of solution on micro domain
    $[-\mu', \mu']$, where $\mu' = 5\varepsilon$ and
    $\varepsilon = 2\cdot10^{-3}$, to solution on a 10 times larger
    domain. Left: x component of the expected solution, middle:
    error with periodic boundary conditions for
    $\m^\varepsilon - \m_\mathrm{init}$, right: error with Dirichlet
    boundary condition.}
  \label{fig:edn}
\end{figure}
In both cases, one can observe errors propagating into the domain
from both boundaries as time increases, even though the amplitude of
the errors is influenced by the type of condition.  Since we cannot
remove this problem even when considering more involved boundary
conditions, we choose to solve the micro problem on a domain
$[-\mu', \mu']^d$, for some $\mu' \ge \mu$, with Dirichlet boundary
conditions and only average over $[-\mu, \mu]^d$. The size of the
domain extension $\mu' - \mu$ together with the time parameter
$\eta$ determine how large the boundary error $E_{\mu'}$ in
\cref{eq:E_avg} becomes. For larger values of $\eta$, we expect a
larger $\mu' - \mu$ to be required to obtain $E_{\mu}$ below a given
threshold, since given a longer final time the errors can propagate
further into the domain. This is investigated in more detail in the
next section.

\subsection{Size of micro domain}\label{sec:micro_size}
Here, we investigate how to choose the size of the micro problem
domain. There are three important parameters that have to be set,
$\mu$ and $\eta$ as in \Cref{thm:paper2}, which determine the size
of the averaging domain in space and time, as well as the outer box
size $\mu'$. Note that the optimal choice of all three parameters is
dependent on the given initial data and the material coefficient.

To determine the influence of the respective parameters, we consider
the example (EX2), with a periodic material coefficient, and
investigate for one micro problem the error $E_\mathrm{avg}$ as
given in \cref{eq:E_avg}. Throughout this section, we consider
averaging kernels $K, K^0$ with $p_x = p_t = 3$ and $q_x = q_t = 7$,
based on the experiments in \cite{paper2}. Typically,
$\varepsilon = 1/400$ is used, which results a value of
$E_\varepsilon$ that is relatively low compared to other
contributions to $E_\mathrm{avg}$.

\subsubsection*{Averaging domain size, $\mu$}
To begin with, we choose a large value for the computational domain $\mu'$,
so that
\[E_\mathrm{avg} \approx E_\varepsilon + E_\eta + E_\mu.\]
We then vary the averaging parameter $\mu$,
which affects the error contribution $E_{\mu}$, that satisfies
\cref{eq:err_terms}, repeated here for convenience,
\begin{equation}
  \label{eq:e_mu}
  E_\mu \le C \left(\mu^{p_x+1} +
    \left(\frac{\varepsilon}{\mu}\right)^{q_x+2}\right).
\end{equation} Based on
the considerations in \cite{paper2} and \Cref{thm:paper2}, we expect
that $\mu$ should be chosen to be a multiple of $\varepsilon$, which
is the scale of the fast spatial oscillations in the problem. 
With the given averaging kernel, the first term on the right-hand side in \cref{eq:e_mu} then is
small in comparison to the other error contributions and the second
term dominates $E_\mu$.

In \Cref{fig:err_mu}, the development of $E_\mathrm{avg}$ when
increasing $\mu$ for (EX2) is shown for several values of $\eta$ and
$\alpha$.  One can observe that as $\mu$ is increased, the error
decreases rapidly from high initial levels. This is due to the
contribution $C \left(\frac{\varepsilon}{\mu}\right)^{q_x+2}$ to
$E_\mu$.  Once $\mu$ becomes sufficiently large, in this example
around $\mu \approx 3.5\varepsilon$, the error does not change
significantly anymore but stays at a constant level, depending on
$\eta$ and $\alpha$. Here $E_\mu$ no longer dominates the error,
which will instead be determined by $E_\eta$ and
$E_\varepsilon$. One can observe that longer times $\eta$ result in
lower overall errors. Furthermore, the errors in the high damping
case, $\alpha = 1$, are considerably lower than with $\alpha = 0.1$
or $\alpha = 0.01$.  However, note that the required value of $\mu$
until the errors no longer decrease is independent of both $\alpha$
and $\eta$.  In the subsequent investigations, we therefore choose a
fixed value of $\mu$ which is slightly above this number, thus
making sure that $E_\mu$ does not significantly influence the
overall error observed there.

\begin{figure}[h]
  \centering
  \includegraphics[width=\textwidth]{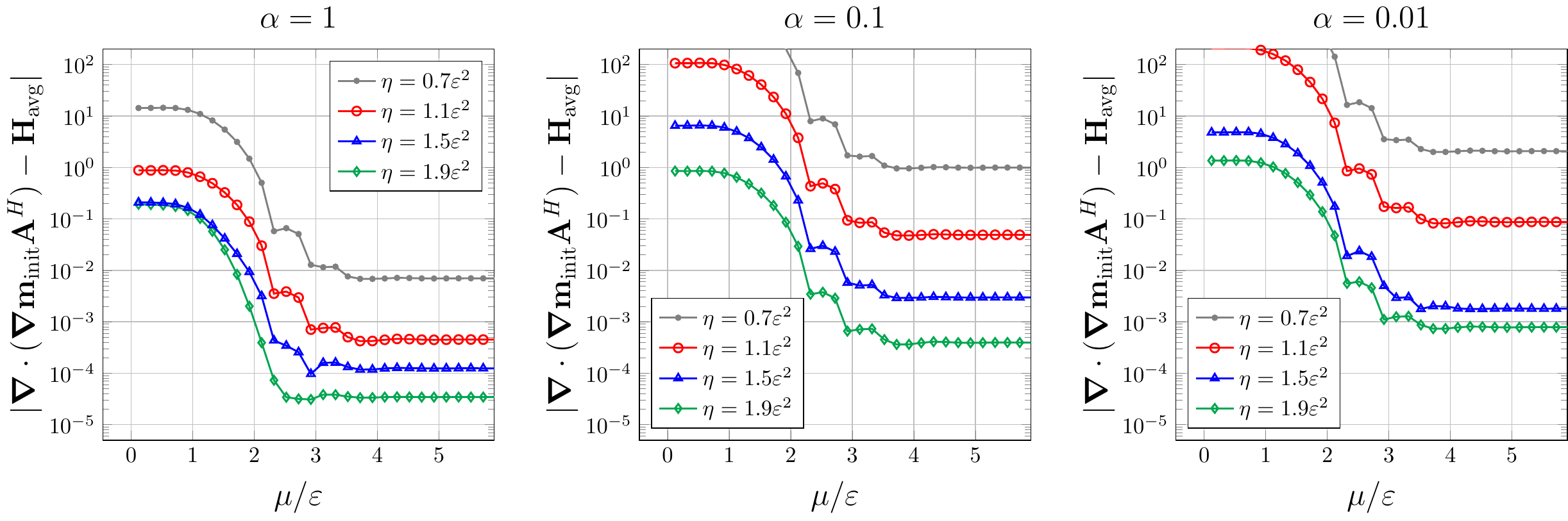}
  \caption{(EX2): Influence of spatial averaging size $\mu$ on the
    overall error in one micro problem for several values of $\eta$
    and $\alpha$ when $\varepsilon = 1/400$. Kernel parameters
    $p_x = p_t = 3$ and $q_x = q_t = 7$.  The outer box size $\mu'$
    is chosen sufficiently large to not significantly influence the
    results.  }
  \label{fig:err_mu}
\end{figure}

\subsubsection*{Full domain size, $\mu'$}
Next, we study the effect of the size of the computational domain,
which is determined by $\mu'$, on the error $E_\mathrm{avg}$.
We here choose $\mu$ large enough so that
\[E_\mathrm{avg} \approx E_\varepsilon + E_\eta + E_{\mu'}.\] We
consider the same choices of $\alpha$ and $\eta$-values as in the
previous example, and vary $\mu'$ to investigate $E_{\mu'}$.  Note
that in contrast to the other error terms, we do not have a model
for $E_{\mu'}$.  As shown in \Cref{fig:err_outer}, a value of $\mu'$
that is only slightly larger than $\mu$ gives a high error, which
decreases as $\mu'$ is increased, until the same error levels as
in \Cref{fig:err_mu}, determined by $E_\eta$ and $E_\varepsilon$, are reached.
\begin{figure}[h]
  \centering
  \includegraphics[width=\textwidth]{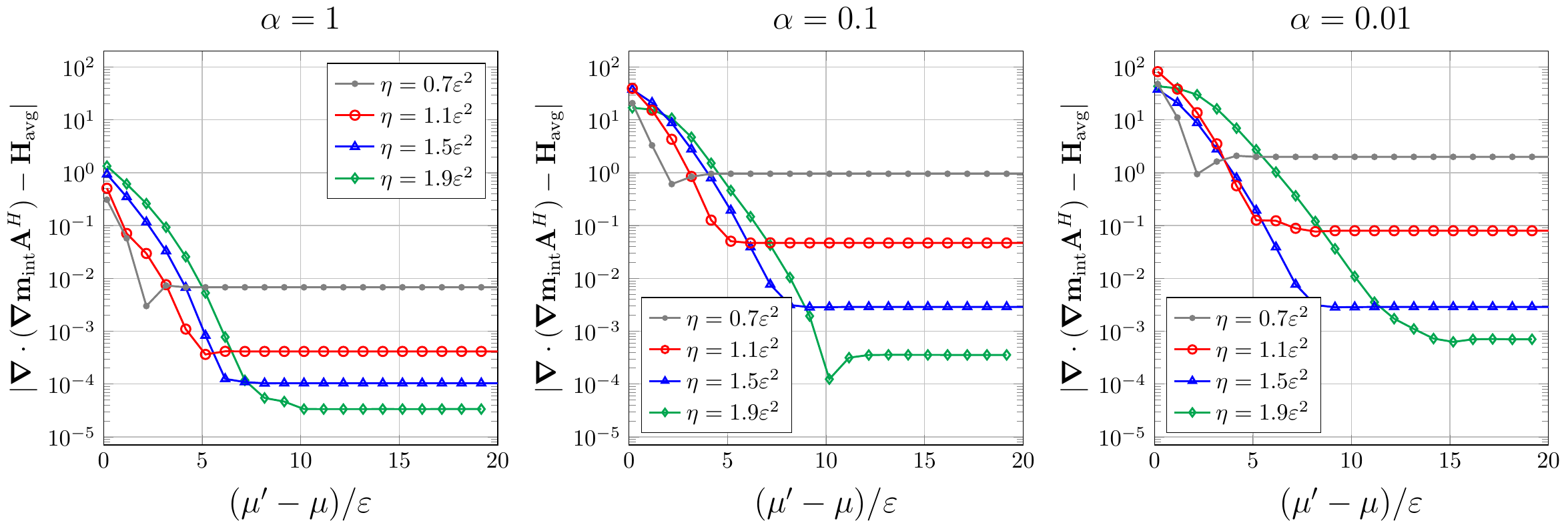}
  \caption{(EX2): Influence of the extension of the computational
    domain, $[-\mu', \mu']^d$ beyond the averaging domain, $[-\mu, \mu]^d$, on the overall averaging error $E_\mathrm{avg}$ in one micro problem for
    several values of $\eta$ and $\alpha$. Here
    $\mu = 3.9 \varepsilon$ and $\varepsilon = 1/400$. 
  }
  \label{fig:err_outer}
\end{figure}
The longer the time interval $\eta$ considered, the larger the domain has to
be chosen to reduce the boundary error $E_{\mu'}$ such that it no
longer dominates $E_\mathrm{avg}$. This is due to the fact that the
boundary error propagates further into the domain the longer time
passes.  We can furthermore observe that larger $\alpha$ results in
somewhat faster convergence of the error for higher values of
$\eta$.

\subsubsection*{Length of time interval $\eta$}
Finally, we consider the influence of $\eta$ and the corresponding
error contribution $E_\eta$ to the averaging error $E_\mathrm{avg}$
as given in \cref{eq:E_avg}. Based on \Cref{thm:paper2},
we have
\begin{align}\label{eq:E_eta}
  E_\eta \le C_\mu \left(\eta^{p_t+1} +
    \left(\frac{\varepsilon^2}{\eta}\right)^{q_t +1}\right),
\end{align}
repeated here for convenience. We consider
$\eta \sim \varepsilon^2$.  With the given choice of averaging
kernel, with $p_t = 3$,  the first term in \cref{eq:E_eta} is small compared to the
second one. We choose the parameters $\mu$ and $\mu'$ such that
\[E_\mathrm{avg} \approx E_\varepsilon + E_\eta\]
and vary $\eta$.
In \Cref{fig:err_eta}, one can then observe that
higher values of $\eta$ result in lower errors, since the second
term on the right hand side in \cref{eq:E_eta} decreases as $\eta$
increases.  This matches with the error behavior depicted in
\Cref{fig:err_mu} and \Cref{fig:err_outer}.  \Cref{fig:err_eta}
furthermore shows that the error eventually saturates at a certain
level, 
corresponding to $E_\varepsilon$. Comparing the errors for
$\alpha = 1$ with $\varepsilon = 1/200$ and $\varepsilon = 1/400$,
one finds that the respective $E_\varepsilon$ differ by a factor of
approximately four, which indicates that
$E_\varepsilon \le C \varepsilon^2$ here.  The different cases of
$\alpha$ considered in \Cref{fig:err_eta} have a similar overall
behavior of the error, but for high damping, $\alpha = 1$, the
development happens for considerably lower values of $\eta$ than in
the other cases. Moreover, in case of $\alpha = 0.01$, we observe
some oscillations as the error decreases.
\begin{figure}[h!]
  \centering
    \begin{subfigure}[t]{.45\textwidth}
      \centering
       \vspace{0pt}
       \includegraphics[width=\textwidth]{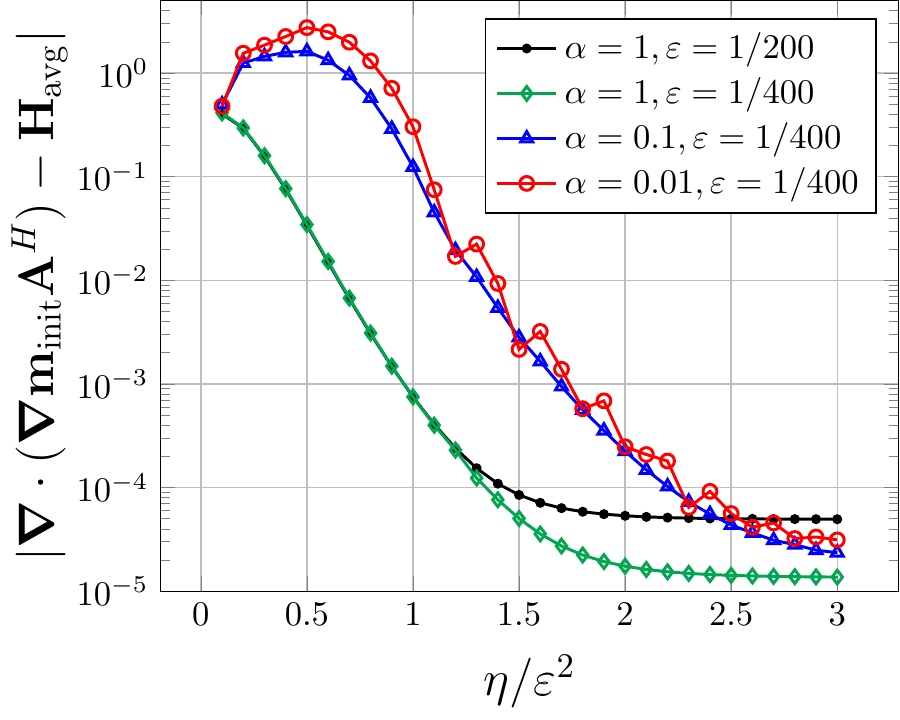}
       \caption{Averaging error $E_\mathrm{avg}$ when varying $\eta$.}
    \label{fig:err_eta}
  \end{subfigure}
  ~
  \begin{subfigure}[t]{.45\textwidth}
    \centering
     \vspace{0pt}
     \includegraphics[width=\textwidth]{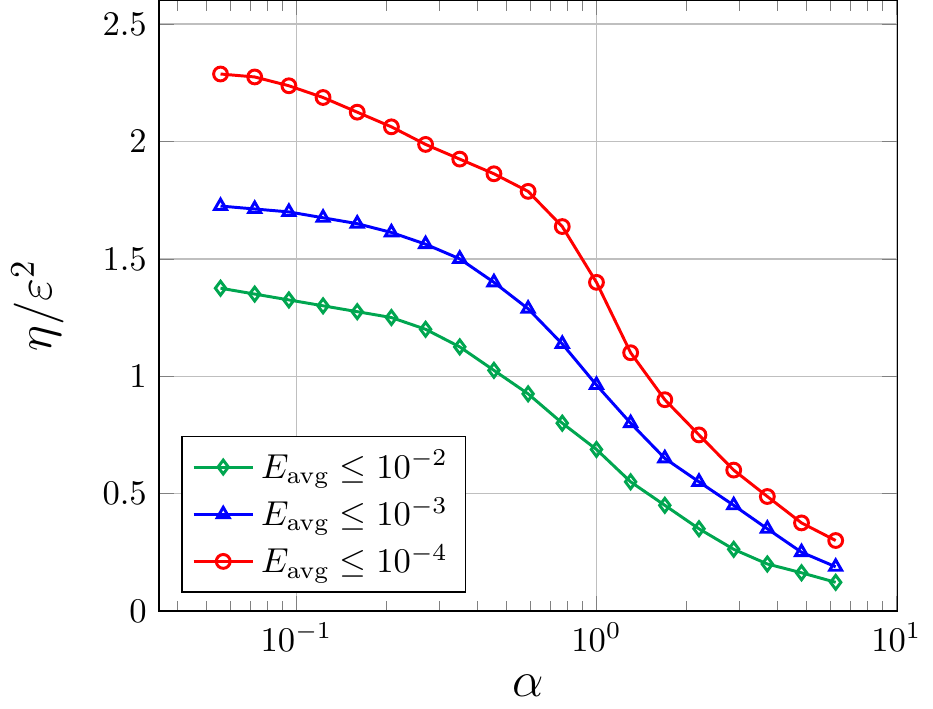}
     \vspace{-4pt}
    \caption{Time required until $E_\mathrm{avg}$ falls below given
      thresholds when varying $\alpha$.}
    \label{fig:err_eta_alfa}
  \end{subfigure}
  \caption{(EX2): Influence of the time averaging length $\eta$ on
    the overall error in one micro problem.  Here
    $\mu = 3.9 \varepsilon$ and $\mu'$ is chosen sufficiently
    big to not significantly change the results. Kernel parameters $p_t = 3$ and $q_t = 7$.
  }
\end{figure}

In \Cref{fig:err_eta_alfa}, we further investigate the influence of
the damping parameter on the time $\eta$ it takes for the averaging
error $E_\mathrm{avg}$ to fall below certain given thresholds. One
can clearly observe that high damping reduces the required time to
reach all three considered error levels. This indicates that the
introduction of artificial damping in the micro problem can help to
significantly reduce computational cost, since a shorter final time
also implies a smaller computational domain as explained in the
previous section. However, since $\alpha \gg 1$ results in a
seriously increased number of time steps necessary to get a stable
solution, as discussed in \Cref{sec:ts_stab}, we conclude that
choosing $\alpha$ around one is most favorable.

\subsubsection*{Example (EX3)}
To support the considerations regarding the choice of micro
parameters, we furthermore study the Landau-Lifshitz problem
\cref{eq:prob} with the setup (EX3).  In (EX3) the material coefficient
has a higher average and higher maximum value than in (EX2). This
results in a higher ``speed'' of the dynamics. In \Cref{fig:err_xy},
the influence of $\mu$, $\mu'$, $\eta$ and $\alpha$, respectively,
on the averaging error $E_\mathrm{avg}$ are shown for this example.

\begin{figure}[h!]
  \centering
     \begin{subfigure}[t]{.35\textwidth}
       \includegraphics[width=\textwidth]{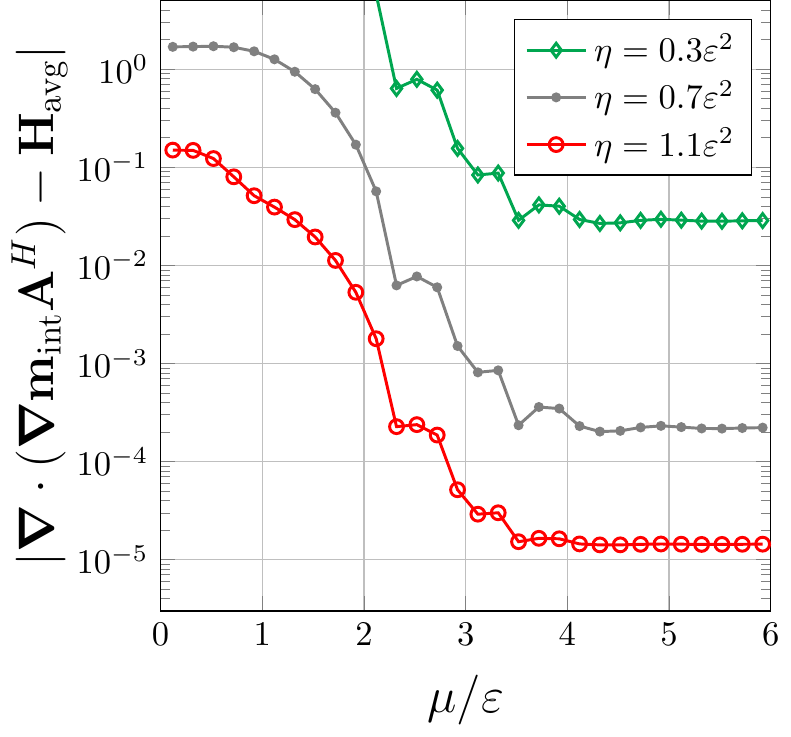}
       \caption{Influence of averaging domain parameter $\mu$, compare to \Cref{fig:err_mu}.
         \label{fig:xy_mu}
       }
   \end{subfigure}
   \hspace{1cm}
   \begin{subfigure}[t]{.35\textwidth}
     \includegraphics[width=\textwidth]{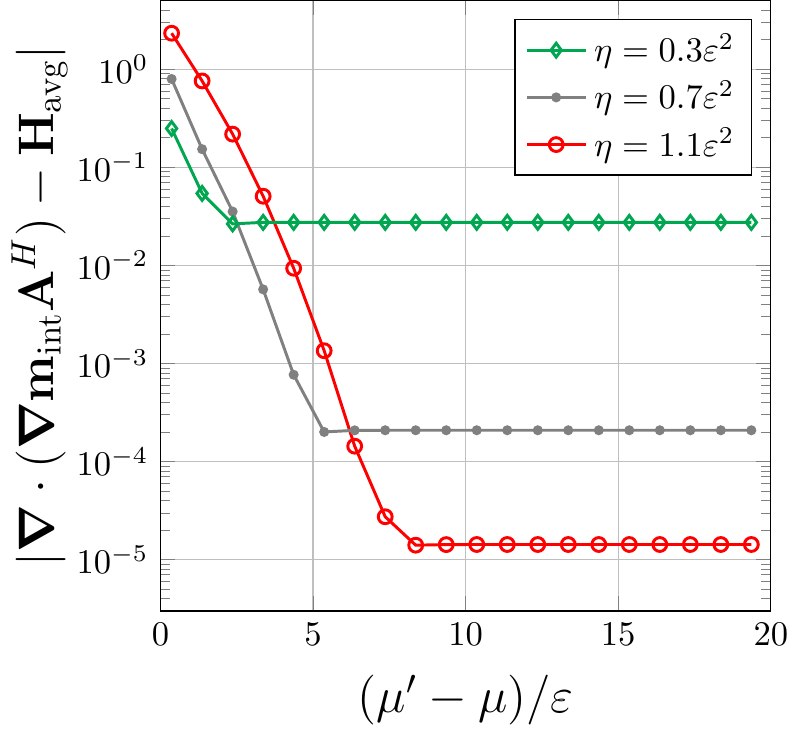}
     \caption{Influence of computational domain extension $\mu'-\mu$, see \Cref{fig:err_outer}.
     \label{fig:xy_box}}
 \end{subfigure}

     \begin{subfigure}[c]{.35\textwidth}
     \includegraphics[width=\textwidth]{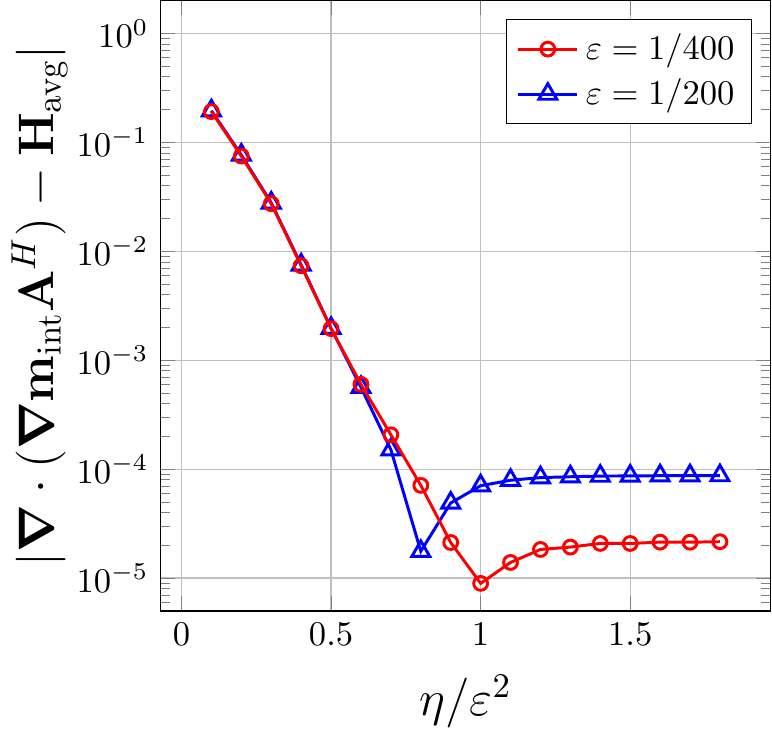}
     \caption{Influence of averaging time $\eta$, compare to \Cref{fig:err_eta}.\label{fig:xy_eta}}
   \end{subfigure}
   \hspace{1cm}
   \begin{subfigure}[c]{.35\textwidth}
     \includegraphics[width=\textwidth]{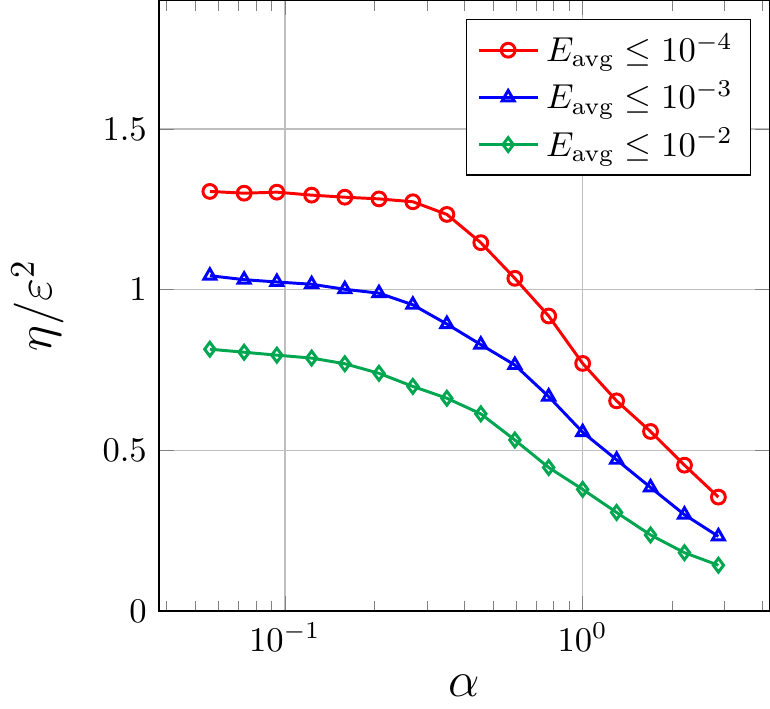}
     \vspace{-5.5pt}
     \caption{Influence of damping $\alpha$, compare to \Cref{fig:err_eta_alfa}. \label{fig:xy_alpha}}
   \end{subfigure}

   \caption{(EX3): Influence of micro domain parameters on error
     $E_\mathrm{avg}$. Parameters not explicitly given are chosen to
     not influence $E_\mathrm{avg}$ significantly. Moreover,
     $\varepsilon = 1/400$, and  $\alpha = 1$ in (a) - (c).  }
  \label{fig:err_xy}
\end{figure}

Qualitatively, the results for (EX3) are the same as (EX2), but some details differ.
A notable difference between the examples (EX2) and (EX3) is that
the time $\eta$ required for saturation of the errors is
considerably shorter in (EX3), as can be observed when comparing
\Cref{fig:xy_eta} to \Cref{fig:err_eta}. An explanation for this is
that due to the faster dynamics in (EX3), comparable effects are
achieved at shorter times.  The ratio between the times $\eta$ it
takes in (EX2) and (EX3), respectively, to reach the level where the
error no longer changes matches approximately with the ratio of the
maxima of the material coefficients.

Moreover, a slightly larger $\mu$ is
required to reach the level where the errors saturate in (EX3).  The
(EX3) saturation errors for a specific value of $\eta$ are lower,
though, due to the fact that the error decreases faster with $\eta$
as discussed previously.

When it comes to the full size of the domain required for
the boundary error to not influence the overall error in a
significant way, one can observe somewhat larger required value of
$\mu'$ in (EX3) when comparing \Cref{fig:xy_box} to
\Cref{fig:err_outer}. This is partly due to the fact that the
overall error for a given $\eta$ is lower in (EX3). Moreover, due to
the faster dynamics, the errors at time $\eta$ have propagated
further into the domain. As a result, the computational domain has
to be chosen approximately the same size in (EX2) and (EX3) to
obtain a certain error level, even though the required $\eta$ is
smaller in (EX3).

\subsection{Computational Cost}\label{sec:cost}
The computational cost per micro problem is a major factor for an
efficient HMM implementation.  Given a spatial discretization with a
certain number of grid points, $K$, per wave length $\varepsilon$,
that is a micro grid spacing $\delta x = \varepsilon/K$, we have in total
$N^d = \left(2 K \mu'/\varepsilon\right)^d$ micro grid points.
The time step size for the micro time integration has to be chosen
as $\delta t \le C_\mathrm{stab, \alpha} \delta x^2$, hence the
number of time steps becomes $M  \ge C_\mathrm{stab, \alpha}^{-1} K^2 \eta / \varepsilon^2$.
It then holds for the computational cost per micro problem that
\begin{align}
  \mathrm{micro~cost} \sim M N^d \sim \frac{1}{C_{stab, \alpha_\mathrm{micro}}} \frac{\eta}{\varepsilon^2} \left(\frac{\mu'}{\varepsilon}\right)^d K^{2+d}.
\end{align}
It is important to note that due to the choice of parameters
$\mu' \sim \varepsilon$ and $\eta \sim \varepsilon^2$, the
computational cost per micro problem is independent of
$\varepsilon$.  This makes it possible to use HMM also for problems
where the computational cost of other approaches, resolving the fast
oscillations, becomes tremendously high.

In general, choosing higher values for $\eta, \mu'$ or $K$ results
in higher computational cost per micro problem. We therefore aim to choose these
values as low as possible without negatively affecting the overall error.
The overall cost is determined by the cost per micro problem and the
choice of the macro discretization size $\Delta X$,
\[\mathrm{cost} \sim \frac{1}{C_\mathrm{stab, \alpha_\mathrm{macro}}} (\Delta X)^{-(2+d)} \mathrm{micro~cost}.\]
This shows the importance of choosing $\Delta X$ relatively large,
wherefore it is advantageous that the overall method proposed based on
MPEA is fourth order accurate, as discussed in \Cref{sec:macro_ts}.
Since all the HMM micro problems are independent of each other, it
is moreover very simple to parallelize their computations. This can be an
effective way to reduce the overall run time of the method.

\subsection{Choice of overall setup}\label{sec:overall}
According to \cref{eq:E_approx}, the overall approximation error in
$\H_\mathrm{avg}$ is
\[E_\mathrm{approx} = E_\mathrm{avg} + E_\mathrm{disc},\] where
$E_\mathrm{disc}$  is determined by the macro discretization size
$\Delta X$ as given in
\cref{eq:E_interpol}. For a
given $\Delta X$, we therefore aim to choose the parameters
$\mu, \mu'$ and $\eta$ so that $E_\mathrm{avg}$ matches
$E_\mathrm{disc}$. Further reducing $E_\mathrm{avg}$ only increases
the computational cost per micro problem without significantly
improving the overall error.

The specific values of the discretization error depend on the given
macro solution and macro location. We here take as an example the
macro initial data and the micro problem solved to obtain
$\H_\mathrm{avg}$ at macro location $(0,0)$.  We consider
$\Delta X = 1/(12\cdot2^i)$, $i = 0, 1, 2$ and suggest in
\Cref{tab:e1} choices for $\eta$ and $\mu'$ in the example setups
(EX2) and (EX3) such that $E_\mathrm{avg}$ is slightly below the
corresponding values for $E_\mathrm{disc}$ when using fourth order
interpolation.
The averaging parameter $\mu$ is fixed to a value
such that $E_\mu$ does not significantly increase $E_\mathrm{avg}$ but not
much higher. This helps to reduce the number of parameters to
vary. Moreover, choosing lower $\mu$ results in a rather steep increase of
$E_\mathrm{avg}$ in comparison to how much the computational cost is
reduced.
We furthermore choose $\alpha = 1$ here to make it simple to compare the
suggestions to the values shown in \Cref{fig:err_eta,fig:err_outer,fig:err_mu}
as well as
\Cref{fig:err_xy}. The optimal choice for $\alpha$ would be slightly
higher.

\begin{table}[h]
  \centering

  \begin{tabular}[h]{c | c || c | c | c || c | c | c}
    \multicolumn{2}{c}{} & \multicolumn{3}{c}{(EX2)} & \multicolumn{3}{c}{(EX3)} \\
    \hline
    $\Delta X$ & $E_\mathrm{disc}$ & $\eta/\varepsilon^2$ & $a_\mathrm{max} \eta/\varepsilon^2$ & $(\mu'-\mu)/\varepsilon$   & $\eta/\varepsilon^2$ & $a_\mathrm{max} \eta/\varepsilon^2$ & $(\mu'-\mu)/\varepsilon$ \\
    \hline
    1/12 & $2.3\cdot10^{-2}$ & 0.7 & 1.09 & 4  & 0.4 & 1.02 & 3 \\
    1/24 & $1.6\cdot10^{-3}$ & 1   & 1.56 & 6 & 0.6 & 1.53  & 5 \\
    1/48 & $1\cdot10^{-4}$   & 1.4 & 2.18 & 8 & 0.8 & 2.05  & 7
  \end{tabular}
  \caption{Example micro problem setups with $\alpha = 1$, where $\mu = 3.9\varepsilon$ for (EX2) and $\mu = 4.2\varepsilon$ for (EX3) .}
  \label{tab:e1}
\end{table}

Based on the values in \Cref{tab:e1}, one can conclude that the
required sizes of the computational domains are very similar between
the considered examples. When scaled by the maximum of the
respective material coefficients, also the suggested final times are
comparable.

To further test the influence of the micro domain setup and
corresponding $E_\mathrm{avg}$ on the overall error, we consider
(EX2) on a unit square domain with periodic boundary condition,
$\alpha = 0.01$ in the original problem and a final time $T =
0.1$. We use artificial damping and set $\alpha = 1.2$ in the micro
problem.  As in the previous examples, we choose averaging kernel
parameters $p_x = p_t = 3$ and $q_x = q_t = 7$ and let
$\varepsilon= 1/400$.  We again fix $\mu = 3.9 \varepsilon$, and run
HMM with MPEA for the macro time stepping to approximate the
homogenized reference solution $\M_0$ at time $T$ for varying
$\Delta X$. \begin{figure}[h]
   \centering
   \begin{subfigure}[c]{.5\textwidth}
     \includegraphics[width=\textwidth]{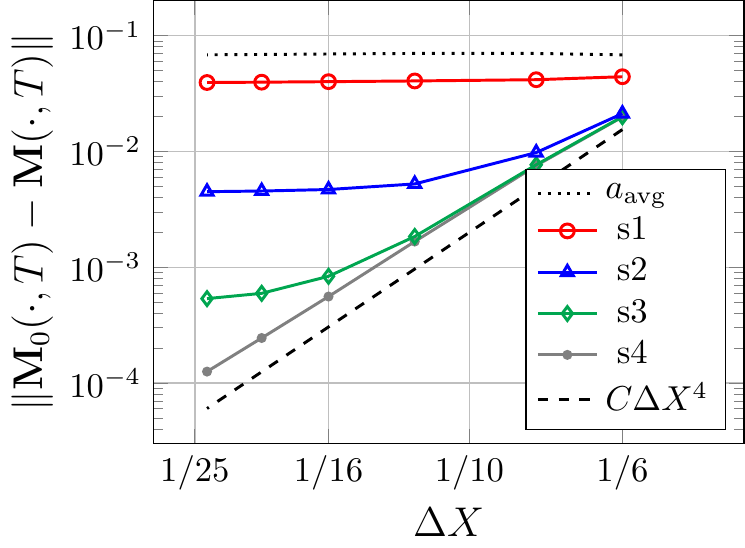}
   \end{subfigure}
   \hspace{3mm}
   \begin{subfigure}[c]{.4\textwidth}
  \begin{tabular}[h]{c | c | c | c }
     & $\eta/\varepsilon^2$ & $\mu'$ &  $E_\mathrm{avg}$ \\
    \hline
    s1 &  0.15 & 4 & $\approx 3\cdot 10^{-1}$  \\
    s2 &  0.45 & 5.5 & $\approx 3\cdot 10^{-2}$ \\
    s3 &  0.7  & 7.5 & $\approx 3\cdot 10^{-3}$ \\
    s4 &  1    & 10 & $\approx 3\cdot 10^{-4}$
  \end{tabular}
  \caption*{Micro problem setups for (EX2) (with $\alpha = 1.2$ and
    $\mu = 3.9$) and resulting $E_\mathrm{avg}$ in micro problem at
    macro location $x_0 = (0,0)$.}
  \label{tab:e2}
   \end{subfigure}
   \caption{$L^2$-norm of difference between HMM solution $\M$ and
     homogenized reference solution $\M_0$ at time $T = 0.1$ for
     (EX2). \label{fig:hmm_err_ex3}}
\end{figure}
Four different combinations of $\eta$ and $\mu'$ are used for the
micro problem, referred to as (s1)-(s4). The resulting $L^2$-norms
of the errors $\M_0 - \M$ are shown in \Cref{fig:hmm_err_ex3},
together with the error one obtains when using the average
$a_\mathrm{avg}$ of the material coefficient $a^\varepsilon$ to
approximate $\A^H$ when solving \cref{eq:hom}. Using
$a_\mathrm{avg}$ can be seen as a naive approach to dealing with the
fast oscillations in the material coefficient. It does in general
not result in good approximations. The corresponding error is
included here to give a baseline for the relevance of the HMM
solutions.

We find that with the micro problem setup (s1), corresponding to a
rather high averaging error, HMM results in a solution that is only
slightly better than the one that is obtained using the average of
the material coefficient.  However, when applying setups with lower
averaging errors, lower overall errors are achieved.  In particular,
with (s4), the setup with the lowest considered averaging error, the
overall error in \Cref{fig:hmm_err_ex3} is determined by the error
$E_\mathrm{disc}$, proportional to $(\Delta X)^4$ since fourth order
interpolation is used to obtain the initial data for the micro
problems.  For the other two setups, (s2) and (s3), the overall
errors saturate at levels somewhat lower than the respective values
of $E_\mathrm{avg}$, corresponding to $e_\mathrm{HMM}$ in
\cref{eq:e_macro}.
Note that this saturation occurs for relatively high values of
$\Delta X$.

\section{Further numerical examples} \label{sec:num_ex}

To conclude this article, we consider several numerical examples
with material coefficients that are not fully periodic. Those cases
are not covered by the theorems in \Cref{sec:hmm}, however, the HMM
approach still results in good approximations. In the 2D examples,
we again include the solution obtained when using a (local) average
of the material coefficient as an approximation to the effective
coefficient to stress the relevance of the HMM solutions.  As for
the periodic examples, we use artificial damping in the micro
problem.

\subsection*{Locally periodic 1D example}

We first consider a one-dimensional example with material coefficient
\begin{equation}
  \label{eq:loc}
  a^\varepsilon(x) = 1.1 + \tfrac{1}{4}\sin(2 \pi x + 1.1) +
  \tfrac{1}{2}\sin(2 \pi x/\varepsilon).
\end{equation}
This coefficient is locally periodic. We consider \cref{eq:prob}
with this coefficient and $\varepsilon = 1/400$, $\alpha = 0.01$ on
the unit interval with periodic boundary conditions.  A comparison
between the solution $\M^\varepsilon$ at time $T = 0.1$, obtained
using a direct numerical simulation resolving the
$\varepsilon$-scale, and corresponding HMM approximation on a coarse
grid with $\Delta X = 1/24$ is shown in \Cref{fig:loc_per}. Here the
HMM parameters are chosen to be $\mu = 3.9 \varepsilon$,
$\mu' = 8 \varepsilon$ and $\eta = 0.9 \varepsilon^2$.  Artificial
damping with $\alpha = 1.2$ is used for the micro problem.  The
averaging kernel parameters are again $p_x = p_t = 3$ and
$q_x = q_t = 7$.
\begin{figure}[h!]
  \centering
  \includegraphics[width=.75\textwidth]{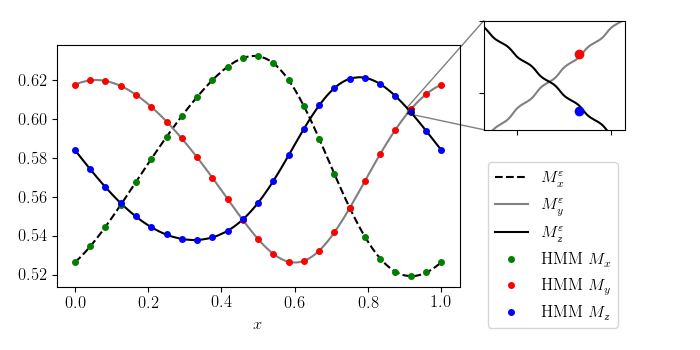}
  \caption{HMM solution to \cref{eq:prob} with $a^\varepsilon$ as in \cref{eq:loc} with
    $\varepsilon = 1/400$ and corresponding
    $\M^\varepsilon$ obtained using direct simulation resolving
    $\varepsilon$,  at $T = 0.1$.}
  \label{fig:loc_per}
\end{figure}
For the direct simulation, we use $\Delta x = 1/6000$, which
corresponds to 15 grid points per $\varepsilon$, and MPEA for time
integration.

One can clearly observe that the HMM solution is very close to the
solution obtained with a direct simulation resolving $\varepsilon$.
Moreover, note that in this example the computation time for HMM is about 15
seconds\footnote{on a computer with Intel i7-4770 CPU at 3.4 GHz},
while the direct simulation takes almost two hours.


\subsection*{Quasi-periodic 2D example}
Next, \cref{eq:prob} is solved in two space dimensions and with
material coefficient
\begin{align}\label{eq:coeff_quasi}
  a^\varepsilon(x) = (1 + 0.25\sin(2\pi x_1/\varepsilon))(1 + 0.25\sin(2\pi x_2/\varepsilon) + 0.25\sin(2\pi r x_2/\varepsilon)),
\end{align}
where $r = 1.41$ as an approximation to $\sqrt{2}$. This coefficient
is periodic in $x_1$-direction but not in $x_2$-direction. If we choose
$\varepsilon = 0.01$, though, it is periodic also in $x_2$ direction
over the whole domain $[0, 1]^2$ but not on the micro domains.  The
initial data is set as in (EX2).

To make direct numerical simulation feasible, we consider the case
$\varepsilon = 0.01$ and set $\Delta x = 1/1500$ in the direct
simulation. For HMM, the micro problem parameters are chosen to be
$\mu = 6.5\varepsilon$, $\eta = 0.7\varepsilon^2$ and
$\mu' = 9\varepsilon$. We use again averaging kernels with
$p_x = p_t = 3$ and $q_x = q_t = 7$ as well as artificial damping
with $\alpha = 1.2$ in the micro problem. On the macro scale,
$\Delta X = 1/16$. The final time is set to $T = 0.2$.

In \Cref{fig:quasi_cont}, the $x$-components of $\M^\varepsilon$,
obtained using a direct simulation resolving $\varepsilon$ and a HMM
solution to \cref{eq:prob} with material coefficient
\cref{eq:coeff_quasi} are shown. Moreover, the solution obtained
when simply using the average of $a^\varepsilon(x)$ as an
approximation is included. One can observe that the HMM solution
captures the characteristics of the overall solution well, while the
approach with an averaged coefficient does not.  To further stress
this, cross sections of the respective solutions at $x_1 = 0.5$ and
$x_2=0.5$ are shown in \Cref{fig:quasi_cross}.

\begin{figure}[hp!]
  \centering
    \begin{subfigure}[b]{\textwidth}
      \centering
      \includegraphics[width=.9\textwidth]{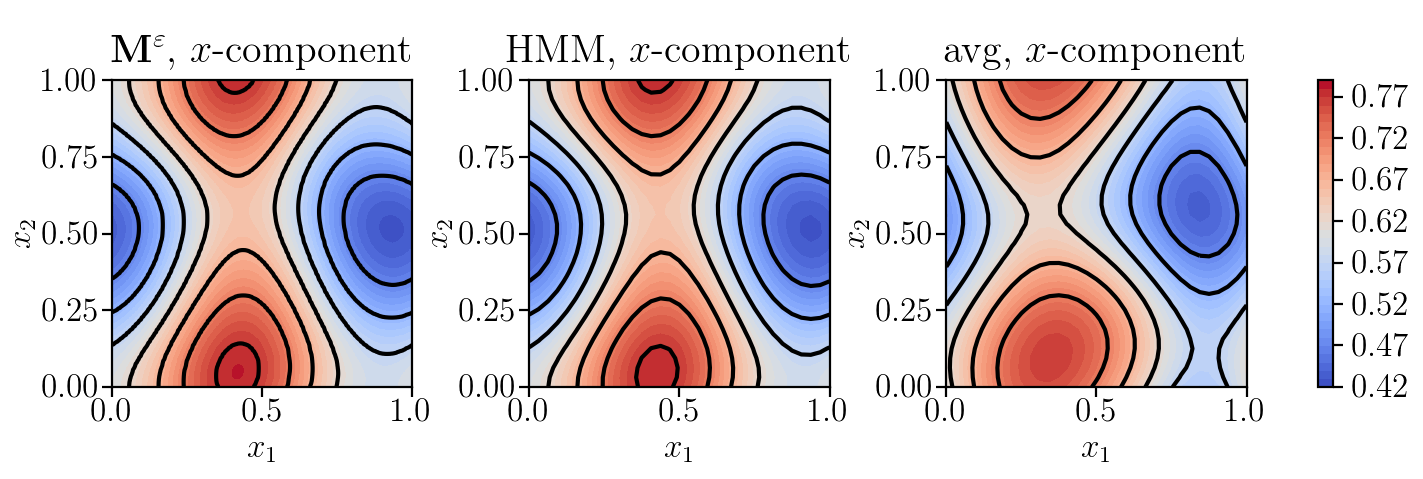}
    \caption{Contours}
     \label{fig:quasi_cont}
   \end{subfigure}
   \begin{subfigure}[b]{\textwidth}
     \centering
     \includegraphics[width=.32\textwidth]{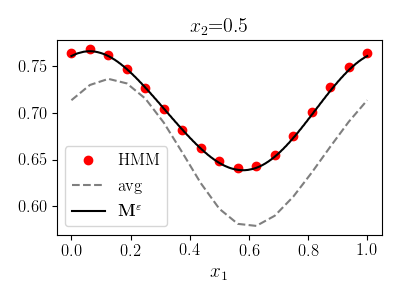}
     \qquad
     \includegraphics[width=.32\textwidth]{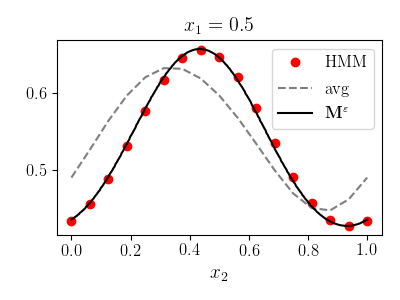}
    \caption{Cross sections}
     \label{fig:quasi_cross}
   \end{subfigure}
  \caption{Quasi-periodic example, with $\varepsilon = 0.01$, $T = 0.2$ and $\alpha = 0.01$.}
  \label{fig:quasi}
\end{figure}

Despite the choice of a rather high $\varepsilon$-value,
$\varepsilon = 0.01$, the direct simulation of this problem took
about 5 days. In comparison, the computational time of HMM was about
4 hours, which is independent of $\varepsilon$.

\subsection*{Locally periodic 2D example}
Finally, we consider a locally periodic 2D example with material coefficient
\begin{align}\label{eq:coeff_loc2D}
  a^\varepsilon(x) = 0.25 \exp\left(-\cos(2\pi(x_1+x_2)/\varepsilon) + \sin(2\pi x_1/\varepsilon)\cos(2\pi x_2)\right).
\end{align}
In this example, we set $\alpha = 0.1$ and choose a final time $T = 0.05$.

The HMM parameters are set to $\mu = 5 \varepsilon$, $\mu' = 7 \varepsilon$ and $\eta = 1.1 \varepsilon^2$. Again
$\alpha = 1.2$ in the micro problem. Initial data and averaging
parameters are set as in the previous example.
Direct simulation solution $\M^\varepsilon$, HMM approximation and a solution based on local
averages of $a^\varepsilon$ are shown in \Cref{fig:loc}.  Also for
this problem HMM captures the characteristics of the solution well,
in contrast to the averaging based solution.

\begin{figure}[H]
  \centering
    \begin{subfigure}[b]{\textwidth}
      \centering
      \includegraphics[width=.9\textwidth]{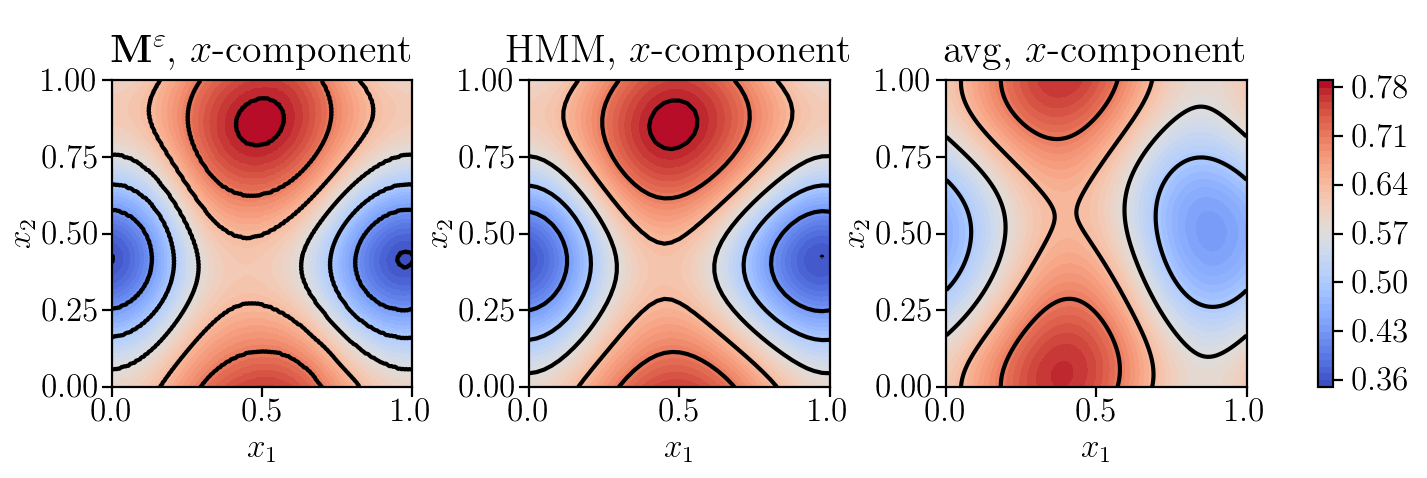}
    \caption{Contours}
     \label{fig:loc_cont}
   \end{subfigure}
   \begin{subfigure}[b]{\textwidth}
     \centering
     \includegraphics[width=.32\textwidth]{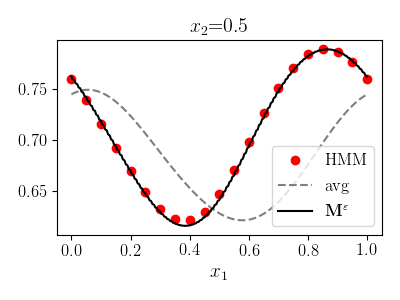}
     \qquad
     \includegraphics[width=.32\textwidth]{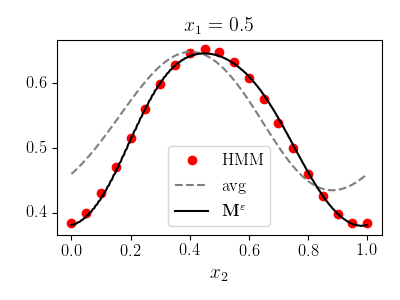}
    \caption{Cross sections}
     \label{fig:loc_cross}
   \end{subfigure}
  \caption{Locally periodic example, with $\varepsilon = 0.01$, $T = 0.05$ and $\alpha = 0.1$.}
  \label{fig:loc}
\end{figure}




\bibliographystyle{acm}
\bibliography{hmm}

\end{document}